\documentclass{amsart}
\newcommand{\intM}[1]{\int_{\Sigma}{{#1}d\mu}}
\newcommand{\llll}[1]{\left\|{#1}\right\|}
\newcommand{\oo}[1]{\left({#1}\right)}
\newcommand{\cc}[1]{\left[{#1}\right]}
\newcommand{\ben}{\begin{eqnarray*}}
\newcommand{\een}{\end{eqnarray*}}
\newcommand{\inner}[1]{\left\langle{#1}\right\rangle}

\newcommand{\norm}[1]{\left|{#1}\right|}

\newcommand{\intcurve}[1]{\int_{\gamma}{{#1}\,ds}}

\newcommand{\R}{\mathbb{R}}
\newcommand{\N}{\mathbb{N}}
\renewcommand{\S}{\mathbb{S}}

\newcommand{\vn}[1]{\lVert #1 \rVert}
\newcommand{\SL}{\mathcal{L}}

\usepackage{amssymb}
\usepackage{url}
\usepackage{amscd}
\allowdisplaybreaks[4]

\theoremstyle{plain}
\newtheorem{theorem}{Theorem}

\newtheorem{corollary}[theorem]{Corollary}
\newtheorem{lemma}[theorem]{Lemma}
\newtheorem{proposition}[theorem]{Proposition}

\theoremstyle{remark}

\begin{document}

\date{Draft}

\title{The polyharmonic heat flow of closed plane curves}

\author[S. Parkins]{Scott Parkins}
\address{Institute for Mathematics and its Applications, School of Mathematics and Applied Statistics\\
University of Wollongong\\
Northfields Ave, Wollongong, NSW $2500$,\\
Australia}
\email{srp854@uow.edu.au}

\author[G. Wheeler]{Glen Wheeler}
\address{Institute for Mathematics and its Applications, School of Mathematics and Applied Statistics\\
University of Wollongong\\
Northfields Ave, Wollongong, NSW $2500$,\\
Australia}
\email{glenw@uow.edu.au}

\thanks{The research of the first author was supported by an Australian
Postgraduate Award.  The research of the second author was supported in part by
Discovery Projects DP120100097 and DP150100375 of the Australian Research Council.}


\begin{abstract}
In this paper we consider the polyharmonic heat flow of a closed curve in the
plane.
Our main result is that closed initial data with initially small normalised
oscillation of curvature and isoperimetric defect flows exponentially fast in
the $C^\infty$-topology to a simple circle.
Our results yield a characterisation of the total amount of time during which
the flow is not strictly convex, quantifying in a sense the failure of the
maximum principle.
\end{abstract}

\maketitle


\begin{section}{Introduction}
Let $\gamma_0:\S\rightarrow\R^2$ be a smooth, closed regular immersed plane curve.
Let $p\in\N_0$.
A one-parameter family $\gamma:\S\times\left[0,T\right)\rightarrow\R^2$ satisfying
\begin{equation}
\frac{\partial}{\partial
t}\gamma=\oo{-1}^{p}\kappa_{s^{2p}}\nu\label{FlowDefinition}\tag{$PF_p$}
\end{equation}
is called the $(2p+2)$-th order polyharmonic heat flow of $\gamma_0$, or the
\emph{polyharmonic flow} for short.
Here $s$ is the regular Euclidean arc length
$s\oo{u}=\int_{0}^{u}{\norm{\gamma_{v}}\,du}$ and $\kappa_{s^{2p}}$ is
$2p$ derivatives of the Euclidean curvature $\kappa$ with respect to
arc length:
\[
\kappa_{s^{2p}} := \partial_s^{2p}\kappa := \frac{\partial^{2p}\kappa}{\partial
s^{2p}}.
\]
We take $\nu$ to be a unit normal vector field to $\gamma$ such that $\kappa\nu =
\partial_s^2\gamma$.

If we take $p=0$, then \eqref{FlowDefinition} is the well-studied \emph{curve
shortening flow} made famous by Hamilton, Gage and Grayson \cite{Hamilton2,grayson1989}:
\[
\partial_t\gamma = \kappa\nu\,.
\]
The curve shortening flow is second-order, and, being a nonlinear geometric
heat equation for the immersion $\gamma$, enjoys the maximum principle and its
standard variations (Harnack inequality, comparison/avoidance principles). This
allows for trademark characteristics such as moving immediately from weak
convexity to strong convexity, preservation of convexity, preservation of
embeddedness, and preservation of graphicality.

The curve shortening flow is the $W^{-0,2} = L^2$ gradient flow for length.
Taking the $H^{-1} = W^{-1,2}$ gradient flow for length yields the fourth order
flow termed the \emph{curve diffusion flow}, whose origins lie in material
science \cite{Mullins1}.
Its gradient flow structure was only later discovered by Fife \cite{Fife}.
The qualitative properties mentioned above for the curve shortening flow do not
hold for the curve diffusion flow (and in fact do not hold for any of the flows
\eqref{FlowDefinition} for $p>0$).
We refer the reader to \cite{Blatt2010,EllPaa2001,EscherIto2005,GI1998,GI1999}
for an overview of these interesting phenomena.
We additionally mention numerical examples contributed by Mayer
\cite{privatecommsmayer} of finite-time singularities arising from embedded
initial data (the resolution of this is an open conjecture that to our
knowledge is due to Giga \cite{privatecommsgiga}).

While local well-posedness belongs by now to standard theory (see for example
\cite{Baker,Mantegazza2}), global analysis and qualitative properties of the
flow remain largely unresolved.
Recently, there have been advances in understanding the stability of the curve
diffusion flow about circles, with work of Elliott-Garcke \cite{EG97}
strengthened by the second author in \cite{Wkosc}.
The result of \cite{Wkosc} relies on the blowup criterion discovered by
Dziuk-Kuwert-Sch\"atzle \cite{DKS}.
The core idea of \cite{Wkosc} is to analyse the normalised oscillation of
curvature:
\[
K_{osc}:=L\intcurve{\oo{\kappa-\bar{\kappa}}^{2}}\,.
\]
The key observation for the curve diffusion flow is that $K_{osc}$ is a natural
energy, being both integrable (in time) for any allowable initial data and
whose blowup characterises finite-time blowup in general.
In this article, we prove that $K_{osc}$ remains a natural energy for every
polyharmonic flow, regardless of how large $p$ is.

In the theorem below and for the remainder of the article we assume $p\in\N$.

\begin{theorem}
\label{TM1}
Suppose $\gamma:\mathbb{S}^{1}\times\left[0,T\right)\rightarrow\mathbb{R}^{2}$
solves $\eqref{FlowDefinition}$. Then there exists a constant
$\varepsilon_{0}>0$ depending only on $p$ such that if
\begin{equation}
\label{EQsmallness}
K_{osc}\oo{0}<\varepsilon_{0}\text{  and  }I\oo{0}<e^{\frac{\varepsilon_{0}}{8\pi^{2}}}
\end{equation}
then $\gamma\oo{\mathbb{S}^{1}}$ approaches a round circle exponentially fast
with radius $\sqrt{\frac{A\oo{\gamma_{0}}}{\pi}}$.
\end{theorem}

Although there is a plethora of negative results on the curve diffusion flow
violating positivity over time, there are relatively few results guaranteeing
preservation.
Theorem \ref{TM1} implies that after some \emph{waiting time}, the flow is
uniformly convex and remains forever so.
An estimate for the waiting time for the curve diffusion flow was given in
\cite{Wkosc}.
Here we extend this to each of the \eqref{FlowDefinition} flows.

\begin{proposition}
\label{PN1}
Suppose $\gamma:\mathbb{S}^{1}\times\left[0,T\right)\rightarrow\mathbb{R}^{2}$
solves $\eqref{FlowDefinition}$. If $\gamma(\cdot,0)$ satisfies \eqref{EQsmallness}, then
\[
\SL\{t\in[0,\infty)\,:\,k(\cdot,t)\not>0\}
\le \frac{2}{p+1}\bigg[
                 \bigg( \frac{L(\gamma_0)}{2\pi}\bigg)^{2(p+1)}
               - \bigg( \frac{A(\gamma_0)}{\pi} \bigg)^{p+1}
                 \bigg]
\,.
\]
\end{proposition}

Above we have used $k(\cdot,t) \not> 0$ to mean $k(s_0,t) \le 0$ for at least
one $s_0$.
This estimate is optimal in the sense that the right hand side is zero for a
simple circle.

One may wish to compare this with the case for classical PDE of higher-order,
where exciting progress on eventual positivity continues to be made
\cite{DGK15,FGG08,GG08,GG09}.

The remainder of the present paper is devoted to proving Theorem \ref{TM1} and
Proposition \ref{PN1}.
We cover some basic definitions and integral formulae in Section 2, before
moving on to essential evolution equations for length, area, and curvature in
Section 3.  We study $K_{osc}$ directly in Section 4, obtaining precise control
over $K_{osc}$ in the case where the initial data is sufficiently close in a
weak isoperimetric sense to a circle and has $K_{osc}$ initially smaller than
an explicit constant.
We continue by adapting Dziuk-Kuwert-Sch\"atzle's blowup criterion argument to
\eqref{FlowDefinition} flows (Lemma \ref{LongTimeLemma1}), yielding in Section
5 global existence.
Further analysis gives exponentially fast convergence to a circle with specific
radius dependent on the initial enclosed area.
We finish Section 5 by giving the proof of Proposition \ref{PN1}.

\end{section}

\begin{section}{Preliminaries}
\begin{lemma}\label{PrelimLemma1}
Suppose $\gamma:\mathbb{S}^{1}\times\left[0,T\right)\rightarrow\mathbb{R}^{2}$
solves $\eqref{FlowDefinition}$, and
$f:\mathbb{S}^{1}\times\left[0,T\right)\rightarrow\mathbb{R}$ is a periodic
function with the same period as $\gamma$. Then 
\begin{equation*}
\frac{d}{dt}\intcurve{f}=\intcurve{f_{t}+\oo{-1}^{p+1}f\cdot\kappa\cdot \kappa_{s^{2p}}}.
\end{equation*}
\end{lemma}
\begin{proof}
We first calculate the evolution of arc length. Because $\nu\perp\tau$, it follows from the Frenet-Serret equations that
\begin{align}
\partial_{s}ds&=\partial_{t}\norm{\gamma_{u}}\,du=\partial_{t}\inner{\gamma_{u},\gamma_{u}}^{\frac{1}{2}}\,du\nonumber\\
&=\norm{\gamma_{u}}^{-1}\inner{\partial_{ut}\gamma,\gamma_{u}}\,du=\inner{\partial_{s}\gamma_{t},\tau}\,ds\nonumber\\
&=\inner{\partial_{s}\oo{\oo{-1}^{p}\kappa_{s^{2p}}\cdot\nu},\tau}\,ds=\oo{-1}^{p}\kappa_{s^{2p}}\inner{\partial_{s}\nu,\tau}\,ds\nonumber\\
&=\oo{-1}^{p+1}\kappa\cdot\kappa_{s^{2p}}\,ds.\label{PrelimLemma1,0}
\end{align}

Next, using the fundamental theorem of calculus and $\oo{\ref{PrelimLemma1,0}}$ we have
\begin{align}
\frac{d}{dt}\int_{\gamma}{f\,ds}&=\frac{d}{dt}\int_{0}^{P\oo{t}}{f\oo{u,t}\left|\gamma_{u}\oo{u,t}\right|\,du}\nonumber\\
&=\intcurve{f_{t}}+\intcurve{f\partial_{t}}+P'\oo{t}\cdot\frac{d}{dP\oo{t}}\int_{0}^{P\oo{t}}{f\oo{u,t}\left|\gamma_{u}\oo{u,t}\right|\,du}\nonumber\\
&=\intcurve{f_{t}+\oo{-1}^{p+1}\kappa\cdot\kappa_{s^{2p}}}+P'\oo{t}f\oo{P\oo{t},t}\left|\gamma_{u}\oo{P\oo{t},t}\right|\nonumber\\
&=\intcurve{f_{t}+\oo{-1}^{p+1}\kappa\cdot\kappa_{s^{2p}}}.\label{PrelimLemma1,1}
\end{align}
Here the last line follows from the fact that
\[
P'(t)\norm{\gamma_{u}\oo{P\oo{t},t}}=\Big(\partial_{t}\big(\gamma\oo{P\oo{t},t}-\gamma\oo{0,t}\big)\Big)^\top=0
\]
because $\partial_{t}\gamma$ is purely normal to $\gamma$.
\end{proof}
\end{section}

\begin{section}{Fundamental evolution equations}
\begin{corollary}\label{PrelimCor1}
Suppose $\gamma:\mathbb{S}^{1}\times\left[0,T\right)\rightarrow\mathbb{R}^{2}$
solves $\eqref{FlowDefinition}$ Then
\[
\frac{d}{dt}L=-\int_{\gamma}{\kappa_{s^{p}}^{2}\,ds}\text{  and  }\frac{d}{dt}A=0.
\]
In particular, the isoperimetric ratio decreases in absolute value with velocity
\[
\frac{d}{dt}I=-\frac{2I}{L}\intcurve{\kappa_{s^{p}}^{2}}\leq0.
\]
\end{corollary}
\begin{proof}
Applying Lemma $\ref{PrelimLemma1}$ with $f\equiv1$ gives the statement for $L$:
\[
\frac{d}{dt}L=\frac{d}{dt}\intcurve{}=\oo{-1}^{p+1}\intcurve{\kappa\cdot \kappa_{s^{2p}}}=-\int_{\gamma}{\kappa_{s^{p}}^{2}\,ds}\leq0.
\]
Here we have performed integration by parts $p$ times. For the statement regarding area, we first state the Frenet-Serret formulas with no torsion:
\begin{equation}
\tau_{s}=\kappa\nu\text{  and  }\nu_{s}=-\kappa\tau.\label{PrelimCor1,1}
\end{equation}
Using the equations in $\oo{\ref{PrelimCor1,1}}$, we wish to derive a formula for the time derivative of the unit normal $\nu$. We first work out the commutator:
\begin{align}
\partial_{ts}&=\partial_{t}\oo{\partial_{s}}=\partial_{t}\oo{\norm{\gamma_{u}}^{-1}\partial_{u}}=\left|\gamma_{u}\right|^{-1}\partial_{t}\partial_{u}-\left|\gamma_{u}\right|^{-2}\oo{\partial_{t}\left|\gamma_{u}\right|}\partial_{u}\nonumber\\
&=\partial_{st}-\norm{\gamma_{u}}^{-3}\inner{\partial_{u}\gamma_{t},\gamma_{u}}\partial_{u}=\partial_{st}-\inner{\partial_{s}\gamma_{t},\tau}\partial_{s}\nonumber\\
&=\partial_{st}-\inner{\partial_{s}\oo{\oo{-1}^{p}\kappa_{s^{2p}}\cdot\nu},\tau}\partial_{s}\nonumber\\
&=\partial_{st}+\oo{-1}^{p}\kappa\cdot\kappa_{s^{2p}}\partial_{s}.\label{PrelimCor1,2}
\end{align}
We then use $\oo{\ref{PrelimCor1,1}},\oo{\ref{PrelimCor1,2}}$ and the identity $\gamma_{s}=\tau$ to calculate:
\begin{align}
\partial_{t}\tau&=\partial_{ts}\gamma=\partial_{st}\gamma+\oo{-1}^{p}\kappa\cdot\kappa_{s^{2p}}\partial_{s}\gamma\nonumber\\
&=\oo{-1}^{p}\cc{\kappa_{s^{2p+1}}\cdot\nu-\kappa\cdot\kappa_{s^{2p}}\cdot\tau}+\oo{-1}^{p}\kappa\cdot\kappa_{s^{2p}}\cdot\tau\nonumber\\
&=\oo{-1}^{p}\kappa_{s^{2p+1}}\cdot\nu.\label{PrelimCor1,3}
\end{align}
Using the fact that $\nu\perp\tau$ and $\norm{\nu}^{2}=1\implies\partial_{t}\nu\perp\nu$, it then follows from $\oo{\ref{PrelimCor1,3}}$ that
\begin{align}
\partial_{t}\nu&=\inner{\partial_{t}\nu,\tau}\tau=-\inner{\nu,\partial_{t}\tau}\tau\nonumber\\
&=-\inner{\nu,\oo{-1}^{p}\kappa_{s^{2p+1}}\cdot\nu}\tau=\oo{-1}^{p+1}\kappa_{s^{2p+1}}\cdot\tau.\label{PrelimCor1,4}
\end{align}
Applying Lemma $\ref{PrelimLemma1}$ with $f=\inner{\gamma,\nu}$ then gives
\begin{align}
\frac{d}{dt}A&=-\frac{1}{2}\frac{d}{dt}\int_{\gamma}{\inner{\gamma,\nu}\,ds}=-\frac{1}{2}\intcurve{\partial_{t}\inner{\gamma,\nu}+\oo{-1}^{p+1}\inner{\gamma,\nu}\cdot\kappa\cdot\kappa_{s^{2p}}}\nonumber\\
&=-\frac{1}{2}\intcurve{\inner{\oo{-1}^{p}\kappa_{s^{2p}}\cdot\nu,\nu}+\inner{\gamma,\oo{-1}^{p+1}\kappa_{s^{2p+1}}\cdot\tau}+\oo{-1}^{p+1}\inner{\gamma,\nu}\cdot\kappa\cdot\kappa_{s^{2p}}}\nonumber\\
&=-\frac{1}{2}\intcurve{\oo{-1}^{p}\kappa_{s^{2p}}+\oo{-1}^{p+1}\kappa_{s^{2p+1}}\inner{\gamma,\tau}+\oo{-1}^{p+1}\inner{\gamma,\tau_{s}}\kappa_{s^{2p}}}\nonumber\\
&-\frac{1}{2}\intcurve{\oo{-1}^{p}\kappa_{s^{2p}}+\oo{-1}^{p+1}\kappa_{s^{2p+1}}\inner{\gamma,\tau}+\oo{-1}^{p}\cc{\inner{\gamma_{s},\tau}\kappa_{s^{2p}}+\inner{\gamma,\tau}\kappa_{s^{2p+1}}}}\nonumber\\
&=\oo{-1}^{p+1}\intcurve{\kappa_{s^{2p}}}=\oo{-1}^{p+1}\kappa_{s^{2p-1}}\Bigg|_{s=0}^{s=L\oo{\gamma}}\nonumber\\
&=0.\nonumber
\end{align}
Here we have used integration by parts in the third last line. The last step follows from the divergence theorem, and using the periodicity of $\gamma$.

To establish the evolution equaiton for the isoperimetric ratio we simply combine the two established results for $L$ and $A$:
\begin{align}
\frac{\partial}{\partial t}I&=\frac{\partial}{\partial t}\oo{\frac{L^{2}}{4\pi A}}=\frac{1}{4\pi A^{2}}\cc{2AL\frac{\partial}{\partial t}L-L^{2}\frac{\partial}{\partial t}A}\nonumber\\
&=-\frac{2L}{4\pi A}\int_{\gamma}{\kappa_{s^{p}}^{2}\,ds}\nonumber\\
&=-\frac{2I}{L}\int_{\gamma}{\kappa_{s^{p}}^{2}\,ds}\leq0.\nonumber
\end{align}
This completes the proof.
\end{proof}

\begin{lemma}\label{CurvatureLemma1}
Suppose that
$\gamma:\mathbb{S}^{1}\times\left[0,T\right)\rightarrow\mathbb{R}^{2}$ solves
$\eqref{FlowDefinition}$ and
\[
\intcurve{\kappa}\Big|_{t=0}=2\omega\pi.
\]
Then
\[
\intcurve{\kappa}=2\omega\pi
\]
for $t\in\left[0,T\right)$. Moreover, the average curvature $\overline{\kappa}=\frac{1}{L}\intcurve{\kappa}$ increases in absolute value with velocity
\[
\frac{d}{dt}\overline{\kappa}=\frac{2\omega\pi}{L^2}\llll{\kappa_{s^{p}}}_{2}^{2}\geq0.
\]
\end{lemma}
\begin{proof}
We first need to calculate the evolution equation for curvature. Using the definition $\kappa=\inner{\nu,\gamma_{ss}}$ along with previous identities, we have
\begin{align}
\frac{\partial}{\partial t}\kappa&=\frac{\partial}{\partial t}\inner{\nu,\gamma_{ss}}=\inner{\nu_{t},\gamma_{ss}}+\inner{\nu,\partial_{t}\gamma_{ss}}=\inner{\nu_{t},\gamma_{ss}}+\inner{\nu,\partial_{ts}\tau}\nonumber\\
&=\inner{\oo{-1}^{p+1}\kappa_{s^{2p+1}}\cdot\tau,\kappa\cdot\nu}+\inner{\nu,\partial_{st}\tau+\oo{-1}^{p}\kappa\cdot\kappa_{s^{2p}}\cdot\tau_{s}}\nonumber\\
&=\inner{\nu,\partial_{s}\oo{\oo{-1}^{p}\kappa_{s^{2p+1}}\cdot\nu}+\oo{-1}^{p}\kappa^{2}\cdot\kappa_{s^{2p}}\cdot\nu}\nonumber\\
&=\oo{-1}^{p}\oo{\kappa_{s^{2p+2}}+\kappa^{2}\cdot\kappa_{s^{2p}}}.\label{CurvatureLemma1,1}
\end{align}
Then, applying Lemma $\ref{PrelimLemma1}$ with $f=\kappa$ gives us
\begin{equation}
\frac{d}{dt}\intcurve{\kappa}=\intcurve{\kappa_{t}+\oo{-1}^{p+1}\kappa^{2}\cdot\kappa_{s^{2p}}}=0.\label{CurvatureLemma1,3}
\end{equation}
It follows from $\oo{\ref{CurvatureLemma1,3}}$ that the integral $\intcurve{\kappa}$ stays constant on $\left[0,T\right)$. This gives the first assertion of the lemma. For the second assertion, we simply use $\oo{\ref{CurvatureLemma1,3}}$ and Corollary $\ref{PrelimCor1}$ and compute:
\begin{align*}
\frac{d}{dt}\overline{\kappa}&=\frac{d}{dt}\oo{\frac{1}{L}\intcurve{\kappa}}=\frac{1}{L^2}\cc{L\cdot\frac{d}{dt}\intcurve{\kappa}-\intcurve{\kappa}\cdot\frac{d}{dt}L}\\
&=-\frac{2\omega\pi}{L^{2}}\cdot-\intcurve{\kappa_{s^{p}}^{2}}=\frac{2\omega\pi}{L^{2}}\llll{\kappa_{s^{p}}}_{2}^{2}\\
&\geq0.
\end{align*}
This completes the proof.
\end{proof}
\end{section}
\begin{section}{The Normalised Oscillation of Curvature}
We now introduce a scale-invariant quantity
\[
K_{osc}:=L\intcurve{\oo{\kappa-\bar{\kappa}}^{2}}
\]
which we call the \emph{normalised oscillation of curvature}.

One can deduce from our previous calculations that this quantity is a natural
one, being that for a one parameter family of curves $\gamma_{t}$ that solves
$\eqref{FlowDefinition}$, $K_{osc}\oo{t}$ is a bounded quantity in $L^{1}$
(and in fact is bounded by a quantity that depends on the initial data,
$\gamma_{0}$ and so can be controlled a priori). Indeed, The fact that
$\intcurve{\oo{\kappa-\bar{\kappa}}}=0$ means that we can apply Lemma
$\ref{AppendixLemma1}$, giving
\[
K_{osc}=L\intcurve{\oo{\kappa-\bar{\kappa}}^{2}}\leq L\oo{\frac{L}{2\pi}}^{2}\intcurve{\kappa_{s}^{2}}.
\]
Now the periodicity of $\kappa$ implies that for every $i\geq 1$, $\intcurve{\kappa_{s^{i}}}=0$, so we can apply Lemma $\ref{AppendixLemma1}$ to the right hand side of the above inequality $p$ more times, yielding
\begin{equation}
K_{osc}\leq L\oo{\frac{L}{2\pi}}^{2p}\intcurve{\kappa_{s^{p}}^{2}}=-\frac{L^{2p+1}}{\oo{2\pi}^{2p}}\frac{d}{dt}L=-\frac{1}{2\oo{p+1}\oo{2\pi}^{2p}}\frac{d}{dt}\oo{L^{2\oo{p+1}}}.\label{OscillationOfCurvature1}
\end{equation}
Here we have utilised the evolution of the length functional. We conclude that for any $t\in\left[0,T\right)$
\begin{align}
\int_{0}^{t}{K_{osc}\oo{\tau}\,d\tau}&\leq-\frac{1}{2\oo{p+1}\oo{2\pi}^{2p}}\oo{L^{2\oo{p+1}}\oo{\gamma_{t}}-L^{2p+1}\oo{\gamma_{0}}}\nonumber\\
&\leq\frac{1}{2\oo{p+1}\oo{2\pi}^{2p}}L^{2\oo{p+1}}\oo{\gamma_{0}}.\label{OscillationCurvature2}
\end{align}
We deduce from $\oo{\ref{OscillationCurvature2}}$ that the normalised oscillation of curvature is a priori controlled in $L^{1}$ over the time of existence of the flow. Furthermore, by repeatedly using Lemma $\ref{AppendixLemma2}$ in a similar fashion, one can easily obtain an $L^{1}$ bound for $\llll{\kappa-\bar{\kappa}}_{\infty}^{2}$ over the interval $\left[0,T\right)$. Firstly
\begin{align*}
\llll{\kappa-\bar{\kappa}}_{\infty}^{2}&\leq\frac{L}{2\pi}\intcurve{\kappa_{s}^{2}}\leq \frac{L}{2\pi}\oo{\frac{L}{2\pi}}^{2}\intcurve{\kappa_{s^{2}}^{2}}\\
&\vdots\\
&\leq \frac{L}{2\pi}\oo{\frac{L}{2\pi}}^{2\oo{p-1}}\intcurve{\kappa_{s^{p}}^{2}}=-\frac{L^{2p-1}}{\oo{2\pi}^{2p-1}}\frac{d}{dt}L\\
&=-\frac{1}{2p\oo{2\pi}^{2p-1}}\frac{d}{dt}L^{2p}.
\end{align*} 
Hence for any $t\in\left[0,T\right)$,
\begin{equation}
\int_{0}^{t}{\llll{\kappa-\bar{\kappa}}_{\infty}^{2}\,d\tau}\leq-\frac{1}{2p\oo{2\pi}^{2p-1}}\oo{L^{2p}\oo{\gamma_{t}}-L^{2p}\oo{\gamma_{0}}}\leq\frac{1}{2p\oo{2\pi}^{2p-1}}L^{2p}\oo{\gamma_{0}}.\label{OscillationCurvature3}
\end{equation}
Next we formulate the evolution equation for $K_{osc}$. 
\begin{lemma}\label{CurvatureLemma2}
Suppose $\gamma:\mathbb{S}^{1}\times\left[0,T\right)\rightarrow\mathbb{R}^{2}$
solves $\eqref{FlowDefinition}$. Then 
\begin{align*}
&\frac{d}{dt}\oo{K_{osc}+8\omega^{2}\pi^{2}\ln{L}}+\frac{\llll{\kappa_{s^{p}}}_{2}^{2}}{L}K_{osc}+2L\llll{\kappa_{s^{p+1}}}_{2}^{2}\\
&=L\intcurve{\cc{\oo{\kappa-\bar{\kappa}}^{3}+\bar{\kappa}\oo{\kappa-\bar{\kappa}}^{2}}_{s^{p}}\oo{\kappa-\bar{\kappa}}_{s^{p}}}.
\end{align*}
\end{lemma}
\begin{proof}
We have
\begin{align*}
&\frac{d}{dt}K_{osc}=\frac{d}{dt}L\cdot\intcurve{\oo{\kappa-\bar{\kappa}}^{2}}+L\cdot\frac{d}{dt}\intcurve{\oo{\kappa-\bar{\kappa}}^{2}}\\
&=-\llll{\kappa_{s^{p}}}_{2}^{2}\intcurve{\oo{\kappa-\bar{\kappa}}^{2}}+L\Biggl[2\intcurve{\oo{\kappa-\bar{\kappa}}\kappa_{t}}+\oo{-1}^{p+1}\intM{\oo{\kappa-\bar{\kappa}}^{2}\cdot\kappa\cdot\kappa_{s^{2p}}}\Biggr]\\
&=-\frac{\llll{k_{s^{p}}}_{2}^{2}}{L}K_{osc}+2\oo{-1}^{p}L\intcurve{\oo{\kappa-\bar{\kappa}}\oo{\kappa_{s^{2p+2}}+\kappa^{2}\cdot\kappa_{s^{2p}}}}\\
&+\oo{-1}^{p+1}L\intcurve{\oo{\kappa-\bar{\kappa}}^{2}\cdot\kappa\cdot\kappa_{s^{2p}}}\\
&=-\frac{\llll{\kappa_{s^{p}}}_{2}^{2}}{L}K_{osc}-2L\llll{\kappa_{s^{p+1}}}_{2}^{2}+2\oo{-1}^{p}L\intcurve{\oo{\kappa-\bar{\kappa}}\cdot\kappa^{2}\cdot\kappa_{s^{2p}}}\\
&+\oo{-1}^{p+1}L\intcurve{\oo{\kappa-\bar{\kappa}}^{2}\cdot\kappa\cdot \kappa_{s^{2p}}}.
\end{align*}
Hence
\begin{align*}
&\frac{d}{dt}K_{osc}+\frac{\llll{\kappa_{s^{p}}}_{2}^{2}}{L}K_{osc}+2L\llll{\kappa_{s^{p+1}}}_{2}^{2}\\
&=2\oo{-1}^{p}L\intcurve{\oo{\kappa-\bar{\kappa}}\cdot\cc{\oo{\kappa-\bar{\kappa}}^{2}+2\bar{\kappa}\oo{\kappa-\bar{\kappa}}+\bar{\kappa}^{2}}\cdot \kappa_{s^{2p}}}\\
&+\oo{-1}^{p+1}L\intcurve{\oo{\kappa-\bar{\kappa}}^{2}\cdot\cc{\oo{\kappa-\bar{\kappa}}+\bar{\kappa}}\cdot\kappa_{s^{2p}}}\\
&=\oo{-1}^{p}L\intcurve{\cc{\oo{\kappa-\bar{\kappa}}^{3}+\bar{\kappa}\oo{\kappa-\bar{\kappa}}^{2}+2\bar{\kappa}^{2}\oo{\kappa-\bar{\kappa}}}\oo{\kappa-\bar{\kappa}}_{s^{2p}}}\\
&=L\intcurve{\cc{\oo{\kappa-\bar{\kappa}}^{3}+\bar{\kappa}\oo{\kappa-\bar{\kappa}}^{2}}_{s^{p}}\oo{\kappa-\bar{\kappa}}_{s^{p}}}+2\bar{\kappa}^{2}L\llll{\kappa_{s^{p}}}_{2}^{2}\\
&=L\intcurve{\cc{\oo{\kappa-\bar{\kappa}}^{3}+\bar{\kappa}\oo{\kappa-\bar{\kappa}}^{2}}_{s^{p}}\oo{\kappa-\bar{\kappa}}_{s^{p}}}+\frac{8\omega^{2}\pi^{2}}{L}\llll{\kappa_{s^{p}}}_{2}^{2}\\
&=L\intcurve{\cc{\oo{\kappa-\bar{\kappa}}^{3}+\bar{\kappa}\oo{\kappa-\bar{\kappa}}^{2}}_{s^{p}}\oo{\kappa-\bar{\kappa}}_{s^{p}}}-8\omega^{2}\pi^{2}\frac{d}{dt}\ln{L}
\end{align*}
Here we have used Corollary $\ref{PrelimCor1}$ and Lemma $\ref{CurvatureLemma1}$ in the penultimate step. Rearranging then yields the desired result.
\end{proof}
\begin{lemma}\label{CurvatureLemma3}
\begin{equation}
L\intcurve{\cc{\oo{\kappa-\bar{\kappa}}^{3}+\bar{\kappa}\oo{\kappa-\bar{\kappa}}^{2}}_{s^{p}}\oo{\kappa-\bar{\kappa}}_{s^{p}}}\leq L\oo{c_{1}K_{osc}+c_{2}\sqrt{K_{osc}}}
\end{equation}
for some universal constants $c_{1},c_{2}>0$. Here $c_{i}=c_{i}\oo{p}$.
\end{lemma}
\begin{proof}
The proof follows from an application of a number of interpolation inequalities which can be found in \cite{DKS}. It has been included in the Appendix for the convenience of the reader.
\end{proof}
\begin{corollary}\label{CurvatureCorollary1}
Suppose $\gamma:\mathbb{S}^{1}\times\left[0,T\right)\rightarrow\mathbb{R}^{2}$
solves $\eqref{FlowDefinition}$. Then
\[
\frac{d}{dt}\oo{K_{osc}+8\omega^{2}\pi^{2}\ln{L}}+\frac{\llll{\kappa_{s^{p}}}_{2}^{2}}{L}K_{osc}+L\oo{2-c_{1}K_{osc}-c_{2}\sqrt{K_{osc}}}\llll{\kappa_{s^{p+1}}}_{2}^{2}\leq0.
\]
Here $c_{1}\oo{p},c_{2}\oo{p}$ are the universal constants given in Lemma $\ref{CurvatureLemma3}$. Moreover, if there exists a $T^{*}$ such that for $t\in\left[0,T^{*}\right)$
\begin{equation}
K_{osc}\oo{t}\leq\frac{8c_{1}+2c_{2}^{2}-2c_{2}\sqrt{8c_{1}+c_{2}^{2}}}{4c_{1}^{2}}=2K^{*},\label{CurvatureCorollary1,0}
\end{equation}
then during this time the estimate
\begin{equation}
K_{osc}+8\omega^{2}\pi^{2}\ln{L}+\int_{0}^{t}{K_{osc}\frac{\llll{\kappa_{s}^{p}}_{2}^{2}}{L}\,d\tau}\leq K_{osc}\oo{0}+8\omega^{2}\pi^{2}\ln{L\oo{0}}\label{CurvatureCorollary1,1}
\end{equation}
holds.

\end{corollary}
\begin{proof}
Combining Lemma $\ref{CurvatureLemma2}$ and Lemma $\ref{CurvatureLemma3}$ immediately gives the first result. Using the assumed smallness of $K_{osc}$ then gives the second. 
\end{proof}

Note that although Corollary $\ref{CurvatureCorollary1}$ implies that the
normalised oscillation of curvature remains bounded if initially sufficiently
small, it does not seem to give tight control of the quantity per se, because
we already know that $\ln{L}$ (on the left hand side of
$\oo{\ref{CurvatureCorollary1,1}}$) is decreasing, and so without further
analysis, one might think that $K_{osc}$ could be static in time (or even
worse, \emph{increasing}).

However, note by the isoperimetric inequality that for any closed curve solving
$\eqref{FlowDefinition}$ we have
\[
\frac{L^{2}\oo{\gamma}}{4\pi A\oo{\gamma}}\geq1,\text{  and so  }\frac{1}{L\oo{\gamma}}\leq\frac{1}{\sqrt{4\pi A\oo{\gamma}}}.
\]
It follows that for any $t\in\left[0,T\right)$,
\begin{equation}
\frac{L\oo{\gamma_{0}}}{L\oo{\gamma_{t}}}\leq\frac{L\oo{\gamma_{0}}}{\sqrt{4\pi A\oo{\gamma_{t}}}}=\frac{L\oo{\gamma_{0}}}{\sqrt{4\pi A\oo{\gamma_{0}}}}=\sqrt{I\oo{\gamma_{0}}}.\label{IsoperimetricConsequence}
\end{equation}
Here we have used the fact that by Corollary $\ref{PrelimCor1}$, the enclosed area of our family of immersed curves is static in time.

So, the quantity $\frac{L\oo{\gamma_{0}}}{L\oo{\gamma_{t}}}$ can be controlled over $\left[0,T\right)$ a priori by assuming that $\gamma_{0}$ is ``sufficiently circular''. In particular, since we may choose $\gamma_{0}$ such that $I\oo{\gamma_{0}}$ is as close to $1$ as we wish (and so $\frac{L\oo{\gamma_{0}}}{L\oo{\gamma_{t}}}$ remains close to $1$ as well),  equation $\oo{\ref{CurvatureCorollary1,1}}$ becomes much more appealing because it can be rearranged to give
\begin{equation}
K_{osc}+\int_{0}^{t}{K_{osc}\frac{\llll{\kappa_{s^{p}}}_{2}^{2}}{L}\,d\tau}\leq K_{osc}\oo{0}+8\omega^{2}\pi^{2}\ln{\sqrt{I\oo{0}}}=K_{osc}\oo{0}+4\omega^{2}\pi^{2}\ln\oo{{I\oo{0}}}.\label{OscillationOfCurvature4}
\end{equation}
This of course is an improvement upon  Corollary $\ref{CurvatureCorollary1}$ because it tells us that $K_{osc}$ can not get larger than the right hand side of the inequality. One problem is that this inequality as it stands is only valid whilst $K_{osc}$ satisfies $\oo{\ref{CurvatureCorollary1,0}}$, and it is not clear from $\oo{\ref{OscillationOfCurvature4}}$ that this smallness condition should hold for the duration of the flow.

 A little bit of tweaking will give us tighter control over $K_{osc}$ for the duration of the flow, and we present this result in the following proposition.
\begin{proposition}\label{FlowProp2}
Let $\gamma:\mathbb{S}^{1}\times\left[0,T\right)\rightarrow\mathbb{R}^{2}$
solve $\eqref{FlowDefinition}$. Additionally, suppose that $\gamma_{0}$ is a
simple closed curve with $\omega=1$, satisfying
\[
K_{osc}\oo{0}\leq K^{\star}\,\,\text{and}\,\,I\oo{0}\leq e^{\frac{K^{\star}}{8\pi^{2}}}.
\]
Then
\[
K_{osc}\oo{t}\leq 2K^{\star}\,\,\text{for}\,\,t\in\left[0,T\right).
\]
\end{proposition}
\begin{proof}
Suppose for the sake of contradiction that $K_{osc}$ does \emph{not} remain bounded by $2K^{\star}$ for the duration of the flow. Then we can find a maximal $T^{\star}<T$ such that
\[
K_{osc}\oo{t}\leq 2K^{\star}\text{  for  }t\in\left[0,T^{\star}\right).
\]
Then, by $\oo{\ref{OscillationOfCurvature4}}$, the following identity holds for $t\in\left[0,T^{\star}\right)$:
\begin{equation}
K_{osc}\oo{t}\leq K_{osc}\oo{0}+4\pi^{2}\ln\oo{{I\oo{0}}}\leq K^{\star}+4\pi^{2}\ln\oo{e^{\frac{K^{\star}}{8\pi^{2}}}}=\frac{3K^{\star}}{2}\text{ for }t\in\left[0,T^{\star}\right).\label{FlowProp2,1}
\end{equation}
We have also used the fact that Lemma $\ref{CurvatureLemma1}$ ensures that $\omega=1$ for the duration of the flow.

Taking $t\nearrow T$ in inequality $\oo{\ref{FlowProp2,1}}$ gives $K_{osc}\leq\frac{3K^{\star}}{2}<2K^{\star}$, meaning that by continuity, $K_{osc}\leq 2K^{\star}$ on some larger time interval $\left[0,T^{\star}+\delta\right)$. But $\left[0,T^{\star}\right)$ was chosen to be the largest time interval containing $0$ such that $K_{osc}$ remains bounded by $2K^{\star}$ and so we have arrived at a contradiction. Thus our assumption that $T<T^{\star}$ must have been false, and the result of the proposition follows.
\end{proof}

\begin{corollary}\label{CurvatureCorollary2}
Let $\gamma:\mathbb{S}^{1}\times\left[0,T\right)\rightarrow\mathbb{R}^{2}$
solve $\eqref{FlowDefinition}$. Additionally, suppose that $\gamma_{0}$ is a
simple embedded closed curve satisfying
\[
K_{osc}\oo{0}\leq K^{\star}<\epsilon_{0}\,\,\text{and}\,\,I\oo{0}\leq e^{\frac{K^{\star}}{8\pi^{2}}}\leq e^{\frac{\varepsilon_{0}}{8\pi^{2}}},
\]
where $\epsilon_{0}<32-2\pi^{2}$ is a sufficiently small constant. Then $\gamma$ remains embedded on $\left[0,T\right)$. 
\end{corollary}
\begin{proof}
Suppose $\gamma:\mathbb{S}^{1}\rightarrow\mathbb{R}^{2}$ is a smooth immersed curve with winding number $\omega=1$. From the Gauss-Bonnet theorem and Lemma $\ref{CurvatureLemma1}$, we know that for $t\in\left[0,T\right)$ the winding number of $\gamma_{t}$ remains the same. Therefore the hypothesis of the corollary implies that $\omega=1$ for the duration of the flow. Define $m\oo{\gamma}$ to be the maximum number of times that $\gamma$ intersects itself in any one point. That is,
\[
m\oo{\gamma}:=\sup_{x\in\mathbb{R}^{2}}\norm{\gamma^{-1}\oo{x}}.
\]
By Theorem $16$ from \cite{Wkosc}, $m$ satisfies the following inequality:
\begin{equation}
K_{osc}\oo{\gamma}\geq16m^{2}-4\omega^{2}\pi^{2}=16m^{2}-4\pi^{2}.\nonumber
\end{equation}
Hence
\begin{equation}
m^{2}\leq\frac{1}{16}\oo{K_{osc}\oo{\gamma}+4\pi^{2}}.\label{CurvatureCorollary2,1}
\end{equation}
Proposition $\ref{FlowProp2}$ then tells us that by the hypothesis of the corollary, $K_{osc}$ remains bounded above by $2K^{\star}$ for the duration of the flow. We can assume without loss of generality that $K^{\star}<32-2\pi^{2}\approx12.26$, and so we have $K_{osc}\oo{\gamma}<64-4\pi^{2}$ on $\left[0,T\right)$. Therefore by $\oo{\ref{CurvatureCorollary2,1}}$ we have
\[
m^{2}<\frac{1}{16}\oo{64-4\pi^{2}+4\pi^{2}}=4\,\,\text{for}\,\,t\in\left[0,T\right),
\]
and embeddedness follows immediately.
\end{proof}

\begin{lemma}\label{LongTimeLemma1}
Suppose $\gamma:\mathbb{S}^{1}\times \left[0,T\right)\rightarrow\mathbb{R}^{2}$
is a maximal solution to $\eqref{FlowDefinition}$. If $T<\infty$ then
\[
\intcurve{\kappa^{2}}\geq c\oo{T-t}^{-1/2\oo{p+1}}
\]
for a universal constant $c>0$.
\end{lemma}
\begin{proof}
Deriving an evolution equation for $\intcurve{\kappa^{2}}$ in the same manner as Lemma $\ref{CurvatureLemma2}$ and using an interpolation inequality in the same vein as \cite{DKS} gives us
\begin{equation}
\frac{d}{dt}\intcurve{\kappa^{2}}+\intcurve{\kappa_{s^{p+1}}^{2}}\leq c\oo{p}\oo{\intcurve{\kappa^{2}}}^{2\oo{m+p}+3},\nonumber
\end{equation}
which implies that
\begin{equation}
-\frac{1}{2\oo{p+1}}\cc{\oo{\intcurve{\kappa^{2}}\Big|_{t=t_{1}}}^{-1/2\oo{p+1}}-\oo{\intcurve{\kappa^{2}}\Big|_{t=t_{0}}}^{-1/2\oo{p+1}}}\leq c\oo{t_{1}-t_{0}}.\label{LongTimeLemma1,1}
\end{equation}
for any times $t_{0}\leq t_{1}$. Note that if $\limsup_{t\rightarrow T}\intcurve{\kappa^{2}}=\infty$, then taking $t_{1}\nearrow T$ in $\oo{\ref{LongTimeLemma1,1}}$ and rearranging will prove the lemma. Assume for the sake of contradiction that $\intcurve{\kappa^{2}}\leq\varrho$ for all $t<T$. By using an argument similar to Theorem $3.1$ of \cite{Wkosc}, we are able to show that the inequality 
\[
\llll{\partial_{u}^{m}\gamma}_{\infty}\leq c_{m}\oo{\varrho,\gamma_{0},T}<\infty
\]
holds for every $m\in\mathbb{N}$ up until time $T$. By short time existence we are then able to extend the life of the flow, contradicting the maximality of $\gamma$. Hence our assumption that $\limsup_{t\rightarrow T}\intcurve{\kappa^{2}}<\infty$ must have been incorrect, and so the limit must diverge. We then conclude the desired result of the lemma from $\oo{\ref{LongTimeLemma1,1}}$.
\end{proof}
\end{section}

\begin{section}{Global analysis}
\begin{corollary}\label{LongTimeCorollary1}
Suppose $\gamma:\mathbb{S}^{1}\times\left[0,T\right)\rightarrow\mathbb{R}^{2}$
solves $\eqref{FlowDefinition}$. Additionally, suppose that $\gamma_{0}$ is
a simple closed curve satisfying
\[
K_{osc}\oo{0}\leq K^{\star}\,\,\text{and}\,\,I\oo{0}\leq e^{\frac{K^{\star}}{8\pi^{2}}}.
\]
Then $T=\infty$.
\end{corollary}
\begin{proof}
Suppose for the sake of contradiction that $T<\infty$. Then by Lemma $\ref{LongTimeLemma1}$ we have
\[
\intcurve{\kappa^{2}}\geq c\oo{T-t}^{-1/2\oo{p+3}}
\]
and so in particular,
\begin{equation}
\intcurve{\kappa^{2}}\rightarrow\infty\text{  as  }t\rightarrow T.\label{LongTimeCorollary1,1}
\end{equation}
Next note that $\oo{\ref{IsoperimetricConsequence}}$ gives us an absolute lower bound on the length of $\gamma$:
\[
L\oo{\gamma_{t}}\geq\sqrt{4\pi A\oo{\gamma_{0}}}.
\]
Hence we establish the following following bound on $K_{osc}$:
\[
K_{osc}=L\intcurve{\kappa^{2}}-4\pi^{2}\geq\sqrt{4\pi A\oo{\gamma_{0}}}\intcurve{\kappa^{2}}-4\pi^{2}.
\]
Hence it follows from $\oo{\ref{LongTimeCorollary1,1}}$ that
\[
K_{osc}\oo{t}\rightarrow\infty\,\,\text{as}\,\,t\rightarrow T.
\]
But this directly contradicts the results of Proposition $\ref{FlowProp2}$, and so we conclude that our assumption that $T$ was finite must have been incorrect. Thus $T=\infty$.
\end{proof}
Recall we know that if $\gamma:\mathbb{S}^{1}\times\left[0,T\right)\rightarrow\mathbb{R}^{2}$ satisfies the hypothesis of Corollary $\ref{LongTimeCorollary1}$ then $T=\infty$, and then identity $\oo{\ref{OscillationCurvature2}}$ tells us that
\begin{equation}
K_{osc}\in L^{1}\oo{\left[0,\infty\right)},\,\,\text{with}\,\,\int_{0}^{\infty}{K_{osc}\oo{\tau}\,d\tau}\leq \frac{1}{2\oo{p+1}\oo{2\pi}^{2p}}L^{2\oo{p+1}}\oo{0}.\label{OscillationOfCurvatureL1}
\end{equation}
So we can conclude that the ``tail'' of the function $K_{osc}\oo{t}$ must get small as $t\nearrow\infty$. However, at the present time we have not ruled out the possibility that $K_{osc}$ gets smaller and smaller as $t$ gets large, whilst vibrating with higher and higher frequency, remaining in $L^{1}\oo{\left[0,\infty\right)}$ whilst never actually fully dissipating to zero in a smooth sense. To rule out this from happening, it is enough to show that $\norm{\frac{d}{dt}K_{osc}}$ remains bounded by a universal constant for all time. To do so we will need to first show that $\llll{\kappa_{s^{p}}}_{2}^{2}$ remains bounded. We will address this issue with the following proposition.

\begin{proposition}\label{LongTimeProp1}
Suppose $\gamma:\mathbb{S}^{1}\times\left[0,T\right)\rightarrow\mathbb{R}^{2}$
solves $\eqref{FlowDefinition}$ and is simple. There exists a
$\varepsilon_{0}>0$ (with $\varepsilon_{0}\leq K^{\star}$) such that if
\[
K_{osc}\oo{0}<\varepsilon_{0}\text{  and  }I\oo{0}<e^{\frac{\varepsilon_{0}}{8\pi^{2}}}
\]
then $\llll{k_{s^{p}}}_{2}^{2}$ remains bounded for all time. In particular,
\[
\intcurve{\kappa_{s^{p}}^{2}}\leq\tilde{c}\oo{\gamma_{0}}
\]
for some constant $\tilde{c}\oo{\gamma_{0}}$ depending only upon the initial immersion.
\end{proposition}
\begin{proof}
We first derive the evolution equation for the quantity $\intcurve{\kappa_{s^{p}}^{2}}$. Applying Lemma $\ref{PrelimLemma1}$ along with repeated applications of the formula for the commutator $\cc{\partial_{t},\partial_{s}}$, we have
\begin{align}
\frac{d}{dt}\intcurve{\kappa_{s^{p}}^{2}}&=2\intcurve{\kappa_{s^{p}}\partial_{t}\kappa_{s^{p}}}+\intcurve{\kappa\cdot\kappa_{s^{p}}^{2}\cdot\kappa_{s^{2p}}}\nonumber\\
&=-2\intcurve{\kappa_{s^{2p+1}}^{2}}+2\oo{-1}^{p}\sum_{j=0}^{p}\oo{-1}^{j}\intcurve{\kappa\cdot\kappa_{s^{p-j}}\cdot\kappa_{s^{p+j}}\cdot\kappa_{s^{2p}}}\nonumber\\
&+\oo{-1}^{p+1}\intcurve{\kappa\cdot\kappa_{s^{p}}^{2}\cdot\kappa_{s^{2p}}}\nonumber\\
&=-2\intcurve{\kappa_{s^{2p+1}}^{2}}+2\intcurve{\kappa^{2}\cdot\kappa_{s^{2p}}^{2}}\nonumber\\
&+2\oo{-1}^{p}\sum_{j=0}^{p-1}\oo{-1}^{j}\intcurve{\cc{\oo{\kappa-\bar{\kappa}}+\bar{\kappa}}\oo{\kappa-\bar{\kappa}}_{s^{p-j}}\oo{\kappa-\bar{\kappa}}_{s^{p+j}}\oo{\kappa_{s^{2p}}}}\nonumber\\
&+\oo{-1}^{p+1}\intcurve{\cc{\oo{\kappa-\bar{\kappa}}+\bar{\kappa}}\oo{\kappa-\bar{\kappa}}_{s^{p}}^{2}\oo{\kappa-\bar{\kappa}}_{s^{2p}}}\nonumber\\
&\leq-2\intcurve{\kappa_{s^{2p+1}}^{2}}+2\intcurve{\kappa^{2}\cdot\kappa_{s^{2p}}^{2}}\nonumber\\
&+c\oo{p}\intcurve{\norm{P_{4}^{4p,2p}\oo{\kappa-\bar{\kappa}}}}+c\oo{p}L^{-1}\intcurve{\norm{P_{3}^{4p,2p}\oo{\kappa-\bar{\kappa}}}}.\label{LongTimeProp1,1}
\end{align}
Here $P_{i}^{j,k}\oo{\cdot}$ stands for a polynomial in $\phi$ of the form
\[
P_{i}^{j,k}\oo{\phi}=\sum_{\mu_{1}+\dots+\mu_{i}=j,\mu_{l}\leq k}\partial_{s}^{\mu_{1}}\phi\star\partial_{s}^{\mu_{2}}\phi\star\cdots\star\partial_{s}^{\mu_{i}}\phi.
\]
(See, for example \cite{DKS} for more details).
Using Lemma $\ref{AppendixLemma5}$, it follows that
\[
\intcurve{\norm{P_{4}^{4p,2p}\oo{\kappa-\bar{\kappa}}}}+L^{-1}\intcurve{\norm{P_{3}^{4p,2p}\oo{\kappa-\bar{\kappa}}}}\leq c\oo{p}\oo{K_{osc}+\sqrt{K_{osc}}}\intcurve{\kappa_{s^{2p+1}}^{2}},
\]
Hence inequality $\oo{\ref{LongTimeProp1,1}}$ can be rearranged to read
\begin{equation}
\frac{d}{dt}\intcurve{\kappa_{s^{p}}^{2}}+\oo{2-c\oo{p}\oo{K_{osc}+\sqrt{K_{osc}}}}\intcurve{\kappa_{s^{2p+1}}^{2}}\leq2\intcurve{\kappa^{2}\cdot\kappa_{s^{2p}}^{2}}.\label{LongTimeProp1,2}
\end{equation}
Next we expand the right hand side of $\ref{LongTimeProp1,2}$ and use the Cauchy-Schwarz inequality on the result:
\begin{align}
2\intcurve{\kappa^{2}\cdot\kappa_{s^{2p}}^{2}}&=2\intcurve{\cc{\oo{\kappa-\bar{\kappa}}^{2}+2\bar{\kappa}\oo{\kappa-\bar{\kappa}}+\bar{\kappa}^{2}}\kappa_{s^{2p}}^{2}}\nonumber\\
&\leq4\intcurve{\oo{\kappa-\bar{\kappa}}^{2}\kappa_{s^{2p}}^{2}}+4\bar{\kappa}^{2}\intcurve{\kappa_{s^{2p}}^{2}}.\label{LongTimeProp1,5}
\end{align}
The first term in $\oo{\ref{LongTimeProp1,5}}$ can be estimated easily, using
Lemma \ref{AppendixLemma2} with $f=\kappa_{s^{2p}}$:
\begin{align}
4\intcurve{\oo{\kappa-\bar{\kappa}}^{2}\kappa_{s^{2p}}^{2}}&\leq4\llll{\kappa_{s^{2p}}}_{\infty}^{2}\intcurve{\oo{\kappa-\bar{\kappa}}^{2}}\nonumber\\
&\leq 4\oo{\frac{L}{2\pi}\intcurve{\kappa_{s^{2p+1}}^{2}}}\intcurve{\oo{\kappa-\bar{\kappa}}^{2}}=\frac{2}{\pi}K_{osc}\intcurve{\kappa_{s^{2p+1}}^{2}}.\label{LongTimeProp1,6}
\end{align}
The second term is dealt with by using Lemma $\ref{AppendixLemma0}$
 with $m=2p$:
\begin{align}
4\bar{\kappa}^{2}\intcurve{\kappa_{s^{2p}}^{2}}&=16\pi^{2}L^{-2}\intcurve{\kappa_{s^{2p}}^{2}}\nonumber\\
&\leq16\pi^{2}L^{-2}\oo{\varepsilon L^{2}\intcurve{\kappa_{s^{2p+1}}^{2}}+\frac{1}{4\varepsilon^{2p}}L^{-\oo{4p+1}}K_{osc}}\nonumber\\
&=16\pi^{2}\varepsilon\intcurve{\kappa_{s^{2p+1}}^{2}}+\frac{16\pi^{2}}{4\varepsilon^{2p}}L^{-\oo{4p+3}}K_{osc}.\nonumber
\end{align}
Here, of course $\varepsilon>0$ can be made as small as desired. Letting $\varepsilon^{\star}=16\pi^{2}\varepsilon$ yields
\begin{equation}
4\bar{\kappa}^{2}\intcurve{\kappa_{s^{2p}}^{2}}\leq\varepsilon^{\star}\intcurve{\kappa_{s^{2p+1}}^{2}}+4\pi^{2}\oo{\frac{16\pi^{2}}{\varepsilon^{\star}}}^{2p}L^{-\oo{4p+3}}K_{osc}.\label{LongTimeProp1,7}
\end{equation}
Substituting $\oo{\ref{LongTimeProp1,6}}$ and $\oo{\ref{LongTimeProp1,7}}$ into $\oo{\ref{LongTimeProp1,2}}$ gives
\begin{align}
&\frac{d}{dt}\intcurve{\kappa_{s^{p}}^{2}}+\oo{2-\oo{c\oo{p}+\frac{2}{\pi}+\varepsilon^{\star}}K_{osc}-c\oo{p}\sqrt{K_{osc}}}\intcurve{\kappa_{s^{2p+1}}^{2}}\nonumber\\
&\leq4\pi^{2}\oo{\frac{16\pi^{2}}{\varepsilon^{\star}}}^{2p}L^{-\oo{4p+3}}K_{osc}\leq4\pi^{2}\oo{\frac{16\pi^{2}}{\varepsilon^{\star}}}^{2p}\oo{4\pi A\oo{\gamma_{0}}}^{-\oo{4p+3}/2}K_{osc}.\label{LongTimeProp1,8}
\end{align}
Here we have used the inequality $\oo{\ref{IsoperimetricConsequence}}$ in the last step. Hence Proposition $\ref{FlowProp2}$ tells us that choosing choosing $K_{osc}\oo{0}<\varepsilon_{0}$ for $\varepsilon_{0}>0$ sufficiently small yields the following inequality
\begin{equation}
\frac{d}{dt}\intcurve{\kappa_{s^{p}}^{2}}\leq c\oo{\gamma_{0}}K_{osc}\label{LongTimeProp1,9}
\end{equation}
for some constant $c\oo{\gamma_{0}}$ which only depends upon our initial immersion. This inequality is valid over $\left[0,T\right)$. Note that we have chosen $\varepsilon^{\star}$ to be sufficiently small so that the absorption process is valid in the last step.
Integrating $\oo{\ref{LongTimeProp1,9}}$ while using our $L^{1}$ bound for $K_{osc}$ from $\oo{\ref{OscillationOfCurvatureL1}}$ then yields for any $t\in\left[0,T\right)$ the following inequality:
\[
\intcurve{\kappa_{s^{p}}^{2}}\leq\intcurve{\kappa_{s^{p}}^{2}}\Big|_{t=0}+\frac{c\oo{\gamma_{0}}}{2\oo{p+1}\oo{2\pi}^{2p}}L^{2\oo{p+1}}\oo{\gamma_{0}}\leq \tilde{c}\oo{\gamma_{0}}
\]
for some new constant $\tilde{c}\oo{\gamma_{0}}$ that only depends on the initial immersion. This completes the proof.
\end{proof}

\begin{corollary}\label{LongTimeCorollary2}
Suppose $\gamma:\mathbb{S}^{1}\times\left[0,T\right)\rightarrow\mathbb{R}^{2}$
solves $\eqref{FlowDefinition}$. Then there exists a constant
$\varepsilon_{0}>0$ (with $\varepsilon_{0}\leq K^{\star}$) such that if
\[
K_{osc}\oo{0}<\varepsilon_{0}\text{  and  }I\oo{0}<e^{\frac{\varepsilon_{0}}{8\pi^{2}}}
\]
then $\gamma\oo{\mathbb{S}^{1}}$ approaches a round circle with radius $\sqrt{\frac{A\oo{\gamma_{0}}}{\pi}}$.
\end{corollary}
\begin{proof}
Recall from a previous discussion that to show $K_{osc}\searrow0$, it will be enough to show that $\norm{K_{osc}'}$ is bounded for all time. 

Firstly, by Corollary $\ref{CurvatureCorollary1}$ and Corollary $\ref{LongTimeCorollary1}$ we know that for $\epsilon_{0}>0$ sufficiently small $T=\infty$ and for all time we have the estimate
\[
\norm{\frac{d}{dt}K_{osc}}\leq\oo{\frac{8\pi^{2}-K_{osc}}{L}}\llll{\kappa_{s^{p}}}_{2}^{2}\leq\frac{8\pi^{2}}{\sqrt{4\pi A\oo{\gamma_{0}}}}\llll{\kappa_{s^{p}}}_{2}^{2}\leq\frac{8\pi^{2}}{\sqrt{4\pi A\oo{\gamma_{0}}}}\cdot\tilde{c}\oo{\gamma_{0}}<\infty.
\]
Here we have used also the results of Proposition $\ref{LongTimeProp1}$.
This immediately tells us that $K_{osc}\searrow0$ as $t\nearrow\infty$.
We will denote the limiting immersion by $\gamma_{\infty}$. That is,
\[
\gamma_{\infty}:=\lim_{t\to\infty}\gamma_{t}\oo{\mathbb{S}^{1}}=\lim_{t\to\infty}\gamma\oo{\cdot,t}.
\]
Our earlier equations imply that $K_{osc}\oo{\gamma_{\infty}}\equiv0$. Note that because the isoperimetric inequality forces $L\oo{\gamma_{\infty}}\geq\sqrt{4\pi A\oo{\gamma_{\infty}}}=\sqrt{4\pi A\oo{\gamma_{0}}}>0$, we can not have $L\searrow0$ and so we may conclude that
\begin{equation}
\int_{\gamma_{\infty}}{\oo{\kappa-\bar{\kappa}}^{2}\,ds}=0.\label{LongTimeCorollary2,1}
\end{equation}
It follow from $\oo{\ref{LongTimeCorollary2,1}}$ that $\kappa\oo{\gamma_{\infty}}\equiv C$ for some constant $C>0$ (note that we know $C$ must be positive because it is impossible for a closed curve with constant curvature to possess negative curvature). That is to say, $\gamma_{t}\oo{\mathbb{S}^{1}}$ approaches a round circle as $t\nearrow\infty$. The final statement of the Corollary regarding the radius of $\gamma_{\infty}$ (which we denote $r\oo{\gamma_{\infty}}$) then follows easily because the enclosed area $A\oo{\gamma_{t}}$ is static in time:
\[
r\oo{\gamma_{\infty}}=\sqrt{\frac{A\oo{\gamma_{\infty}}}{\pi}}=\sqrt{\frac{A\oo{\gamma_{0}}}{\pi}}.
\]
\end{proof}
Since the previous corollary tells us that $\gamma_{t}\oo{\mathbb{S}^{1}}\rightarrow\mathbb{S}_{\sqrt{\frac{A\oo{\gamma_{0}}}{\pi}}}^{1}$, we can conclude that for every $m\in\mathbb{N}$ there exists a sequence of times $\left\{t_{j}\right\}$ such that 
\[
\int{\kappa_{s^{m}}^{2}}\Big|_{t=t_{j}}\searrow0.
\]
Unfortunately, this is only subconvergence, and does not allow us to rule out the possibility of short sharp ``spikes'' (oscillations) in time. Indeed, even if we were to show that for every $m\in\mathbb{N}$ we have $\llll{\kappa_{s^{m}}}_{2}^{2}\in L^{1}\oo{\left[0,\infty\right)}$ (which is true), this would not be enough because these aforementioned ``spikes'' could occur on a time interval approaching that of (Lebesgue) measure zero. To overcome this dilemma, we attempt to control $\norm{\frac{d}{dt}\intcurve{\kappa_{s^{m}}^{2}}}$, and show that his quantity can be bounded by a multiple of $K_{osc}\oo{0}$ (which can be fixed to be as small as desired a priori). We will see this allows to strengthen the subconvergences result above to one of classical exponential convergence. 

\begin{corollary}[Exponential Convergence]\label{LongTimeCorollary3}
Suppose $\gamma:\mathbb{S}^{1}\times\left[0,T\right)\rightarrow\mathbb{R}^{2}$
solves $\eqref{FlowDefinition}$ and satisfies the assumptions of Corollary
$\ref{LongTimeCorollary2}$. Then for each $m\in\mathbb{N}$ there are constants
$c_{m},c_{m}^{\star}$ such that we have the estimates
\[
\intcurve{\kappa_{s^{m}}^{2}}\leq c_{m}e^{-c_{m}^{\star}t}\text{  and  }\llll{\kappa_{s^{m}}}_{\infty}\leq \sqrt{\frac{L\oo{\gamma_{0}}c_{m+1}}{2\pi}}e^{-\frac{c_{m+1}^{\star}}{2}t}.
\]
\end{corollary}
\begin{proof}
We first derive the evolution equation for $\intcurve{\kappa_{s^{m}}^{2}},m\in\mathbb{N}$ in a similar manner to Proposition $\ref{LongTimeProp1}$:
\begin{align}
\frac{d}{dt}\intcurve{\kappa_{s^{m}}^{2}}&=-2\intcurve{\kappa_{s^{m+p+1}}^{2}}+2\oo{-1}^{p}\sum_{j=1}^{m}\oo{-1}^{j}\intcurve{\kappa\cdot\kappa_{s^{m-j}}\cdot\kappa_{s^{m+j}}\cdot\kappa_{s^{2p}}}\nonumber\\
&+\oo{-1}^{p}\intcurve{\kappa\cdot\kappa_{s^{m}}^{2}\cdot\kappa_{s^{2p}}}.\label{LongTimeCorollary3,1}
\end{align}
We need to be careful in dealing with the extraneous terms in $\oo{\ref{LongTimeCorollary3,1}}$. We wish to apply Lemma $\ref{AppendixLemma5}$
  but to do so must consider the cases $p\geq m$ and $p\leq m$ separately.
	If $p\geq m$, with $p=m+l,l\in\mathbb{N}_{0}$, then we can perform integration by parts on each term in $\oo{\ref{LongTimeCorollary3,1}}$ $l$ times:
\begin{align}
&2\oo{-1}^{p}\sum_{j=1}^{m}\oo{-1}^{j}\intcurve{\kappa\cdot\kappa_{s^{m-j}}\cdot\kappa_{s^{m+j}}\cdot\kappa_{s^{2p}}}+\oo{-1}^{p}\intcurve{\kappa\cdot\kappa_{s^{m}}^{2}\cdot\kappa_{s^{2p}}}\nonumber\\
&=2\oo{-1}^{p}\sum_{j=1}^{m-1}\oo{-1}^{j}\intcurve{\oo{\kappa-\bar{\kappa}+\bar{\kappa}}\oo{\kappa-\bar{\kappa}}_{s^{m-j}}\oo{\kappa-\bar{\kappa}}_{s^{m+j}}\oo{\kappa-\bar{\kappa}}_{s^{m+p+l}}}\nonumber\\
&+2\oo{-1}^{m+p}\intcurve{\cc{\oo{\kappa-\bar{\kappa}}^{2}+2\bar{\kappa}\oo{\kappa-\bar{\kappa}}+\bar{\kappa}^{2}}\oo{\kappa-\bar{\kappa}}_{s^{2m}}\oo{\kappa-\bar{\kappa}}_{s^{m+p+l}}}\nonumber\\
&+\oo{-1}^{p}\intcurve{\oo{\kappa-\bar{\kappa}+\bar{\kappa}}\oo{\kappa-\bar{\kappa}}_{s^{m}}^{2}\oo{\kappa-\bar{\kappa}}_{s^{m+p+l}}}\nonumber\\
&=2\oo{-1}^{p}\sum_{j=1}^{m-1}\oo{-1}^{j+l}\intcurve{\partial_{s}^{l}\cc{\oo{\kappa-\bar{\kappa}+\bar{\kappa}}\oo{\kappa-\bar{\kappa}}_{s^{m-j}}\oo{\kappa-\bar{\kappa}}_{s^{m+j}}}\oo{\kappa-\bar{\kappa}}_{s^{p+p}}}\nonumber\\
&+2\intcurve{\partial_{s}^{l}\cc{\cc{\oo{\kappa-\bar{\kappa}}^{2}+2\bar{\kappa}\oo{\kappa-\bar{\kappa}}+\bar{\kappa}^{2}}\oo{\kappa-\bar{\kappa}}_{s^{2m}}}\oo{\kappa-\bar{\kappa}}_{s^{m+p}}}\nonumber\\
&+\oo{-1}^{p+l}\intcurve{\partial_{s}^{l}\cc{\oo{\kappa-\bar{\kappa}+\bar{\kappa}}\oo{\kappa-\bar{\kappa}}_{s^{m}}^{2}}\oo{\kappa-\bar{\kappa}}_{s^{m+p}}}\nonumber\\
&\leq c\oo{m,p}\intcurve{\norm{P_{4}^{2\oo{m+p},m+p}\oo{\kappa-\bar{\kappa}}}}+c\cdot L^{-1}\intcurve{\norm{P_{3}^{2\oo{m+p},m+p}\oo{\kappa-\bar{\kappa}}}}\nonumber\\
&+2\bar{\kappa}^{2}\intcurve{\kappa_{s^{2m+l}}\cdot\kappa_{s^{m+p}}}\nonumber\\
&\leq4\pi^{2}L^{-2}\intcurve{\kappa_{s^{m+p}}^{2}}+c\oo{m,p}\oo{K_{osc}+\sqrt{K_{osc}}}\intcurve{\kappa_{s^{m+p+1}}^{2}}.\label{LongTimeCorollary3,2}
\end{align}
Here we have used the energy inequality Lemma \ref{AppendixLemma5} on the $P$-style terms, as well as Lemma $\ref{CurvatureLemma1}$ which tells us that $\bar{\kappa}=\frac{2\pi}{L}$.

If $p<m$ (say with $m=p+l,l\in\mathbb{N}$), then we must proceed slightly differently. In this case, identity $\oo{\ref{LongTimeCorollary3,1}}$ becomes
\begin{align}
&2\oo{-1}^{p}\sum_{j=1}^{m}\oo{-1}^{j}\intcurve{\kappa\cdot\kappa_{s^{m-j}}\cdot\kappa_{s^{m+j}}\cdot\kappa_{s^{2p}}}+\oo{-1}^{p}\intcurve{\kappa\cdot\kappa_{s^{m}}^{2}\cdot\kappa_{s^{2p}}}\nonumber\\
&=2\oo{-1}^{p}\sum_{j=1}^{p+l-1}\oo{-1}^{j}\intcurve{\oo{\kappa-\bar{\kappa}+\bar{\kappa}}\oo{\kappa-\bar{\kappa}}_{s^{p+l-j}}\oo{\kappa-\bar{\kappa}}_{s^{p+l+j}}\oo{\kappa-\bar{\kappa}}_{s^{2p}}}\nonumber\\
&+2\oo{-1}^{m+p}\intcurve{\cc{\oo{\kappa-\bar{\kappa}}^{2}+2\bar{\kappa}\oo{\kappa-\bar{\kappa}}+\bar{\kappa^{2}}}\oo{\kappa-\bar{\kappa}}_{s^{m+p+l}}\oo{\kappa-\bar{\kappa}}_{s^{2p}}}\nonumber\\
&+\intcurve{\oo{\kappa-\bar{\kappa}+\bar{\kappa}}\oo{\kappa-\bar{\kappa}}_{s^{p+l}}^{2}\oo{\kappa-\bar{\kappa}}_{s^{2p}}}\nonumber\\
&=2\oo{-1}^{p}\sum_{j=1}^{p}\oo{-1}^{j}\intcurve{\oo{\kappa-\bar{\kappa}+\bar{\kappa}}\oo{\kappa-\bar{\kappa}}_{s^{p+l-j}}\oo{\kappa-\bar{\kappa}}_{s^{p+l+j}}\oo{\kappa-\bar{\kappa}}_{s^{2p}}}\nonumber\\
&+2\oo{-1}^{p}\sum_{j=p+1}^{p+l-1}\oo{-1}^{j}\intcurve{\oo{\kappa-\bar{\kappa}+\bar{\kappa}}\oo{\kappa-\bar{\kappa}}_{s^{p+l-j}}\oo{\kappa-\bar{\kappa}}_{s^{p+l+j}}\oo{\kappa-\bar{\kappa}}_{s^{2p}}}\nonumber\\
&+2\oo{-1}^{m+p}\intcurve{\cc{\oo{\kappa-\bar{\kappa}}^{2}+2\bar{\kappa}\oo{\kappa-\bar{\kappa}}+\bar{\kappa^{2}}}\oo{\kappa-\bar{\kappa}}_{s^{m+p+l}}\oo{\kappa-\bar{\kappa}}_{s^{2p}}}\nonumber\\
&+\intcurve{\oo{\kappa-\bar{\kappa}+\bar{\kappa}}\oo{\kappa-\bar{\kappa}}_{s^{p+l}}^{2}\oo{\kappa-\bar{\kappa}}_{s^{2p}}}\nonumber\\
&\leq 2\oo{-1}^{m+p}\oo{\frac{2\pi}{L}}^{2}\intcurve{\kappa_{s^{m+p+l}}\cdot\kappa_{s^{2p}}}+c\oo{m,p}\intcurve{\norm{P_{4}^{2\oo{m+p},m+p}\oo{\kappa-\bar{\kappa}}}}\nonumber\\
&+c\oo{m,p}L^{-1}\intcurve{\norm{P_{3}^{2\oo{m+p},m+p}\oo{\kappa-\bar{\kappa}}}}+c\oo{m,p}\intcurve{\norm{P_{4}^{2\oo{m+p},M}\oo{\kappa-\bar{\kappa}}}}\nonumber\\
&+c\oo{m,p}L^{-1}\intcurve{\norm{P_{3}^{2\oo{m+p},M}\oo{\kappa-\bar{\kappa}}}}\nonumber\\
&+2\oo{-1}^{p}\sum_{j=p+1}^{p+l-1}\oo{-1}^{j}\intcurve{\oo{\kappa-\bar{\kappa}+\bar{\kappa}}\oo{\kappa-\bar{\kappa}}_{s^{p+l-j}}\oo{\kappa-\bar{\kappa}}_{s^{p+l+j}}\oo{\kappa-\bar{\kappa}}_{s^{2p}}}\label{LongTimeCorollary3,3}
\end{align}
Here $M:=\max\left\{m,2p\right\}<m+p$.
The second and third terms of $\oo{\ref{LongTimeCorollary3,3}}$ are identical to those in our calculation of $\oo{\ref{LongTimeCorollary3,2}}$, in which we established the identity
\begin{align}
&\intcurve{\norm{P_{4}^{2\oo{m+p},m+p}\oo{\kappa-\bar{\kappa}}}}+L^{-1}\intcurve{\norm{P_{3}^{2\oo{m+p},m+p}\oo{\kappa-\bar{\kappa}}}}\nonumber\\
&\leq c\oo{m,p}\oo{K_{osc}+\sqrt{K_{osc}}}\intcurve{\kappa_{s^{m+p+1}}^{2}}.\label{LongTimeCorollary3,4}
\end{align} 
The fourth and fifth terms in $\oo{\ref{LongTimeCorollary3,3}}$ are estimated in a similar way. Because $M<m+p$, we are free to utilise Lemma $\ref{AppendixLemma5}$ wth $K=m+p+1$ and then the terms are estimatable in the same way as the $P$-style terms in $\oo{\ref{LongTimeCorollary3,2}}$. We conclude that
\begin{align}
&\intcurve{\norm{P_{4}^{2\oo{m+p},M}\oo{\kappa-\bar{\kappa}}}}+L^{-1}\intcurve{\norm{P_{3}^{2\oo{m+p},M}\oo{\kappa-\bar{\kappa}}}}\nonumber\\
&\leq c\oo{m,p}\oo{K_{osc}+\sqrt{K_{osc}}}\intcurve{\kappa_{s^{m+p+1}}^{2}}.\label{LongTimeCorollary3,5}
\end{align}
Finally, the last part of $\oo{\ref{LongTimeCorollary3,3}}$ involving the summation can be estimated by applying integrating by parts by parts $j-p$ times to each term in $j$ and then estimating in the same way as above:
\begin{align}
&2\oo{-1}^{p}\sum_{j=p+1}^{p+l-1}\oo{-1}^{j}\intcurve{\oo{\kappa-\bar{\kappa}+\bar{\kappa}}\oo{\kappa-\bar{\kappa}}_{s^{p+l-j}}\oo{\kappa-\bar{\kappa}}_{s^{p+l+j}}\oo{\kappa-\bar{\kappa}}_{s^{2p}}}\nonumber\\
&2\oo{-1}^{p}\sum_{j=p+1}^{p+l-1}\oo{-1}^{2j-p}\intcurve{\partial_{s}^{j-p}\cc{\oo{\kappa-\bar{\kappa}+\bar{\kappa}}\oo{\kappa-\bar{\kappa}}_{s^{p+l-j}}\oo{\kappa-\bar{\kappa}}_{s^{2p}}}\oo{\kappa_{s^{2p+l}}}}\nonumber\\
&\leq c\oo{m,p}\intcurve{\norm{P_{4}^{2\oo{m+p},m+p}\oo{\kappa-\bar{\kappa}}}}+c\oo{m,p}L^{-1}\intcurve{\norm{P_{3}^{2\oo{m+p},m+p}}\oo{\kappa-\bar{\kappa}}}\nonumber\\
&\leq c\oo{m,p}\oo{K_{osc}+\sqrt{K_{osc}}}\intcurve{\kappa_{s^{m+p+1}}^{2}}.\label{LongTimeCorollary3,6}
\end{align}
Combining $\oo{\ref{LongTimeCorollary3,4}}$,$\oo{\ref{LongTimeCorollary3,5}}$ and $\oo{\ref{LongTimeCorollary3,6}}$ and substituting into $\oo{\ref{LongTimeCorollary3,3}}$ then gives
\begin{align}
&2\oo{-1}^{p}\sum_{j=1}^{m}\oo{-1}^{j}\intcurve{\kappa\cdot\kappa_{s^{m-j}}\cdot\kappa_{s^{m+j}}\cdot\kappa_{s^{2p}}}+\oo{-1}^{p}\intcurve{\kappa\cdot\kappa_{s^{m}}^{2}\cdot\kappa_{s^{2p}}}\nonumber\\
&\leq 4\pi^{2}L^{-2}\intcurve{\kappa_{s^{m+p}}^{2}}+c\oo{m,p}\oo{K_{osc}+\sqrt{K_{osc}}}\intcurve{\kappa_{s^{m+p+1}}^{2}}.\label{LongTimeCorollary3,7}
\end{align}
We can clearly see from $\oo{\ref{LongTimeCorollary3,2}}$ and $\oo{\ref{LongTimeCorollary3,7}}$ that the estimates of the extraneous terms in $\oo{\ref{LongTimeCorollary3,1}}$ are of the same form, regardless of the sign of $p-m$. We can conclude that
\begin{equation}
\frac{d}{dt}\intcurve{\kappa_{s^{m}}^{2}}+\oo{2-c\oo{m,p}\oo{K_{osc}+\sqrt{K_{osc}}}}\intcurve{\kappa_{s^{m+p+1}}^{2}}\leq 4\pi^{2}L^{-2}\intcurve{\kappa_{s^{m+p}}^{2}}.\label{LongTimeCorollary3,8}
\end{equation}

Let us step back for a moment and forget about our time parameter, assuming without loss of generality that we are looking at a fixed time slice.

We claim that for any smooth closed curve $\gamma$ and any $l\in\mathbb{N}$, there exists a universal, bounded constant $c_{l}>0$ such that
\begin{equation}
\intcurve{\kappa_{s^{l}}^{2}}\leq c_{l}L^{2}K_{osc}\intcurve{\kappa_{s^{l+1}}^{2}}.\label{LongTimeCorollary3,9}
\end{equation}
Let us assume for the sake of contradiction that we can not find a suitable constant $c_{l}<\infty$ such that inequality $\oo{\ref{LongTimeCorollary3,9}}$ holds. Then, there exists a sequence of immersions $\left\{\gamma_{j}\right\}$ such that 
\begin{equation}
R_{j}:=\frac{\llll{\kappa_{s^{l}}}_{2,\gamma_{j}}^{2}}{L^{2}\oo{\gamma_{j}}K_{osc}\oo{\gamma_{j}}\llll{\kappa_{s^{l+1}}}_{2,\gamma_{j}}^{2}}\nearrow\infty\text{ as  }j\rightarrow\infty.\label{LongTimeCorollary3,10}
\end{equation}
Now, Theorem $\ref{AppendixTheorem1}$ implies that for any $j\in\mathbb{N}$ we have
\[
R_{j}\leq \frac{\frac{L^{2}\oo{\gamma_{j}}}{4\pi^{2}}\llll{\kappa_{s^{l+1}}}_{2,\gamma_{j}}^{2}}{L^{2}\oo{\gamma_{j}}K_{osc}\oo{\gamma_{j}}\llll{\kappa_{s^{l+1}}}_{2,\gamma_{j}}^{2}}=\frac{1}{4\pi^{2}K_{osc}\oo{\gamma_{j}}},
\]
and so the only way that $\oo{\ref{LongTimeCorollary3,10}}$ can occur is if we have
\begin{equation}
K_{osc}\oo{\gamma_{j}}\searrow0\text{  as  }j\rightarrow\infty.\label{LongTimeCorollary3,11}
\end{equation}
Then, as each $\gamma_{j}$ satisfies the criteria of Theorem $\ref{AppendixTheorem1}$, we conclude there is a subsequence of immersions $\left\{\gamma_{j_{k}}\right\}$ and an immersion $\gamma_{\infty}$ such that $\gamma_{j_{k}}\rightarrow\gamma_{\infty}$ in the $C^{1}$-topology. Moreover, by $\oo{\ref{LongTimeCorollary3,11}}$ we have $K_{osc}\oo{\gamma_{\infty}}=0$. But this implies that $\gamma_{\infty}$ must be a circle, in which case both sides of inequality $\oo{\ref{LongTimeCorollary3,9}}$ are zero. Hence the inequality holds trivially with the immersion $\gamma_{\infty}$ for \emph{any} $c_{l}$ we wish, and so we can not in fact have $R_{j}\nearrow\infty$. This contradicts $\oo{\ref{LongTimeCorollary3,10}}$, and so the assumption that we can not find a constant $c_{l}$ such that the inequality $\oo{\ref{LongTimeCorollary3,9}}$ holds, must be false. 

Next, combining $\oo{\ref{LongTimeCorollary3,9}}$ with $\oo{\ref{LongTimeCorollary3,8}}$ gives us
\[
\frac{d}{dt}\intcurve{\kappa_{s^{m}}^{2}}+\oo{2-c\oo{m,p}\oo{K_{osc}+\sqrt{K_{osc}}}-4\tilde{c}_{m}\pi^{2}K_{osc}}\intcurve{\kappa_{s^{m+p+1}}^{2}}\leq0.
\]
Here $\tilde{c}_{m}$ is our new interpolative constant, which we take to be the largest of all optimal constants in inequality  $\oo{\ref{LongTimeCorollary3,9}}$ for closed simple curves with length bounded by $L\oo{\gamma_{0}}$. Then, because $K_{osc}\rightarrow0$, we know that there exists a time, say $t_{m}$, such that for $t\geq t_{m}$, $c\oo{m,p}\oo{K_{osc}+\sqrt{K_{osc}}}+4\tilde{c}_{m}\pi^{2}K_{osc}\leq1$. Hence for $t\geq t_{m}$ the previous inequality implies that
\begin{equation}
\frac{d}{dt}\intcurve{\kappa_{s^{m}}^{2}}\leq-\intcurve{\kappa_{s^{m+p+1}}^{2}}.\label{LongTimeCorollary3,12}
\end{equation}

Next, applying inequality Lemma $\ref{AppendixLemma1}$ $p+1$ times and using the monotonicity of $L\oo{\gamma_{t}}$ gives
\[
\intcurve{\kappa_{s^{m}}^{2}}\leq\oo{\frac{L^{2}\oo{\gamma_{t}}}{4\pi^{2}}}^{p+1}\intcurve{\kappa_{s^{m+p+1}}^{2}}\leq\oo{\frac{L^{2}\oo{\gamma_{0}}}{4\pi^{2}}}^{p+1}\intcurve{\kappa_{s^{m+p+1}}^{2}}.
\]
Hence if we define $c_{m}^{\star}:=\oo{\frac{4\pi^{2}}{L^{2}\oo{\gamma_{0}}}}^{p+1}$ then we conclude from $\oo{\ref{LongTimeCorollary3,12}}$ that for any $t\geq t_{m}$ we have the estimate
\[
\frac{d\intcurve{\kappa_{s^{m}}^{2}}}{\intcurve{\kappa_{s^{m}}^{2}}}\leq-c_{m}^{\star}\,dt.
\]
Integrating over $\left[t_{m},t\right]$ and exponentiating yields
\[
\intcurve{\kappa_{s^{m}}^{2}}\leq\int_{\gamma_{t_{m}}}{\kappa_{s^{m}}^{2}\,ds}\cdot e^{-c_{m}^{\star}\oo{t-t_{m}}}=\oo{e^{c_{m}^{\star}t_{m}}\int_{\gamma_{t_{m}}}{\kappa_{s^{m}}^{2}\,ds}}\cdot e^{-c_{m}^{\star}t},
\]
which is the first statement of the corollary. For the second statement, we simply combine the first statement and Lemma $\ref{AppendixLemma2}$ with $f=\kappa_{s^{m}}$:
\[
\llll{\kappa_{s^{m}}}_{\infty}^{2}\leq \frac{L\oo{\gamma_{0}}}{2\pi}\intcurve{\kappa_{s^{m+1}}^{2}}\leq \frac{L\oo{\gamma_{0}}c_{m+1}}{2\pi}e^{-c_{m+1}^{\star}t}.
\]
The pointwise exponential convergence result follows immediately from taking the square root of both sides.
\end{proof}

Let us finish by proving Proposition \ref{PN1}.

\begin{proof}
We follow \cite{Wkosc}.
Rearranging $\gamma$ in time if necessary, we may assume that
\begin{align*}
k(\cdot,t) \not> 0 ,\qquad &\text{ for all }t\in[0,t_0)\\
k(\cdot,t) > 0,    \qquad &\text{ for all }t\in[t_0,\infty)
\end{align*}
where 
$t_0 > \frac{2}{p+1}\bigg[
                 \bigg( \frac{L(\gamma_0)}{2\pi}\bigg)^{2(p+1)}
               - \bigg( \frac{A(\gamma_0)}{\pi} \bigg)^{p+1}
                 \bigg]$,
otherwise we have nothing to prove.  However in this case we have
\begin{align*}
\frac{d}{dt}L
 &= -\vn{k_{s^p}}_2^2
 \le -\frac{4\pi^2}{L^2}\vn{k_{s^{p-1}}}_2^2
 \le \cdots
 \le -\bigg(\frac{4\pi^2}{L^2}\bigg)^{p-1}\vn{k_s}_2^2
\\
 &\le -\frac{\pi^2}{L^2}\bigg(\frac{4\pi^2}{L^2}\bigg)^{p-1}\vn{k}_2^2
\\
 &\le -\frac{4\pi^4}{L^3}\bigg(\frac{4\pi^2}{L^2}\bigg)^{p-1},&\text{for}\ t\in[0,t_0),
\intertext{where we used the fact that $\gamma$ is closed and that the curvature has a zero.  This implies}
L^{2p+2}(t) &\le -\frac{p+1}{2}(2\pi)^{2p+2}t + L^{2p+2}(\gamma_0),& \text{for}\ t\in[0,t_0),
\end{align*}
and thus $L^{2p+2}(t_0) < (4\pi A(\gamma_0))^{2p+2}$.
This is in contradiction with the isoperimetric inequality.
\end{proof}

\end{section}

\begin{section}{Appendix}
\begin{lemma}\label{AppendixLemma0}
Let $\gamma:\mathbb{S}^{1}\rightarrow\mathbb{R}^{2}$ be a smooth closed curve with Euclidean curvature $\kappa$ and arc length element $ds$. Then for any $m\in\mathbb{N}$ we have
\[
\intcurve{\kappa_{s^{m}}^{2}}\leq\varepsilon L^{2}\intcurve{\kappa_{s^{m+1}}^{2}}+\frac{1}{4\varepsilon^{m}}L^{-\oo{2m+1}}K_{osc},
\]
where $\varepsilon>0$ can be made as small as desired.
\end{lemma}
\begin{proof}
We will prove the lemma inductively. The case $m=1$ can be checked quite easily, by applying integration by parts and the Cauchy-Schwarz inequality:
\begin{align}
\intcurve{\kappa_{s}^{2}}&=\intcurve{\oo{\kappa-\bar{\kappa}}_{s}^{2}}=-\intcurve{\oo{\kappa-\bar{\kappa}}\oo{\kappa-\bar{\kappa}}_{s^{2}}}\nonumber\\
&\leq\oo{\intcurve{\oo{\kappa-\bar{\kappa}}^{2}}}^{\frac{1}{2}}\oo{\intcurve{\kappa_{s^{2}}^{2}}}^{\frac{1}{2}}\nonumber\\
&\leq\varepsilon L^{2}\intcurve{\kappa_{s^{2}}^{2}}+\frac{1}{4\varepsilon^{1}}L^{-2}\intcurve{\oo{\kappa-\bar{\kappa}}^{2}}.\nonumber
\end{align}

Next assume inductively that the statement is true for $j=m$. That is, assume that
\begin{equation}
\intcurve{\kappa_{s^{j}}^{2}}\leq\varepsilon L^{2}\intcurve{\kappa_{s^{j+1}}^{2}}+\frac{1}{4\varepsilon^{j}}L^{-\oo{2j+1}}K_{osc}\label{AppendixLemma0,1}
\end{equation}
where $\varepsilon>0$ can be made as small as desired.

Again performing integration by parts and the Cauchy-Schwarz inequality, we have for any $\varepsilon>0$:
\begin{align}
\intcurve{\kappa_{s^{j+1}}^{2}}&=-\intcurve{\kappa_{s^{j}}\cdot\kappa_{s^{j+2}}}\leq\oo{\intcurve{\kappa_{s^{j}}^{2}}}^{\frac{1}{2}}\oo{\intcurve{\kappa_{s^j+2}^{2}}}^{\frac{1}{2}}\nonumber\\
&\leq\frac{\varepsilon}{2} L^{2}\intcurve{\kappa_{s^{j+2}}^{2}}+\frac{1}{2\varepsilon}L^{-2}\intcurve{\kappa_{s^{j}}^{2}}.\label{AppendixLemma0,2}
\end{align}
Substituting the inductive assumption $\oo{\ref{AppendixLemma0,1}}$ into $\oo{\ref{AppendixLemma0,2}}$ then gives
\begin{align}
\intcurve{\kappa_{s^{j+1}}^{2}}&\leq\frac{\varepsilon}{2} L^{2}\intcurve{\kappa_{s^{j+2}}^{2}}+\frac{1}{2\varepsilon}L^{-2}\cc{\varepsilon L^{2}\intcurve{\kappa_{s^{j+1}}^{2}}+\frac{1}{4\varepsilon^{j}}\oo{\varepsilon}L^{-\oo{2j+1}}K_{osc}},\nonumber
\end{align}
meaning that
\[
\frac{1}{2}\intcurve{\kappa_{s^{j+1}}^{2}}\leq\frac{\varepsilon}{2} L^{2}\intcurve{\kappa_{s^{j+2}}^{2}}+\frac{1}{2}\cdot\frac{1}{4\varepsilon^{j+1}}L^{-\oo{2\oo{j+1}+1}}K_{osc}.
\]
Multiplying out by $2$ then gives us the inductive step, completing the lemma.
\end{proof}

\begin{lemma}\label{AppendixLemma1}
Let $f:\mathbb{R}\rightarrow\mathbb{R}$ be an absolutely continuous and periodic function of period $P$. Then, if $\int_{0}^{P}{f\,dx}=0$ we have
\[
\int_{0}^{P}{f^{2}\,dx}\leq\frac{P^{2}}{4\pi^{2}}\int_{0}^{P}{f_{x}^{2}\,dx},
\]
with equality if and only if
\[
f\oo{x}=A\cos\oo{\frac{2\pi}{P}x}+B\sin\oo{\frac{2\pi}{P}x}
\]
for some constants $A,B$.
\end{lemma}
\begin{proof}
We will use the calculus of variations. Essentially, we wish to find $f$ that maximises the integral $\int_{0}^{P}{f^{2}\,dx}$, given a fixed value of $\int_{0}^{P}{f_{x}^{2}\,dx}$. We will show that combining this with the requirement that $\int_{0}^{P}{f\,dx}=0$ forces the extremal function to satisfy
\[
\frac{\int_{0}^{P}{f^{2}\,dx}}{\int_{0}^{P}{f_{x}^{2}\,dx}}\leq\frac{P^{2}}{4\pi^{2}}.
\] 
For the constrained problem, the associated Euler-Lagrange equation is
\[
L=f^{2}+\lambda f_{x}^{2},
\]
with extremal functions satisfying
\[
\frac{\partial L}{\partial f}-\frac{d}{dx}\oo{\frac{\partial F}{\partial f_{x}}}=2f-2\lambda f_{xx}=0.
\]
That is to say,
\begin{equation}
f_{xx}-\frac{1}{\lambda}f=0.\label{AppendixLemma1,1}
\end{equation}
This means that
\[
0\leq\int_{0}^{P}{f_{x}^{2}\,dx}=-\int_{0}^{P}{ff_{xx}\,dx}=-\frac{1}{\lambda}\int_{0}^{P}{f^{2}\,dx},
\]
which forces $\lambda<0$. By standard arguments, we conclude from $\oo{\ref{AppendixLemma1,1}}$ that our extremal function is
\begin{equation}
f\oo{x}=A\cos\oo{\frac{x}{\sqrt{\norm{\lambda}}}}+B\sin\oo{\frac{x}{\sqrt{\norm{\lambda}}}}.\label{AppendixLemma1,2}
\end{equation}
Here $A,B$ are constants. The periodicity of $f$ forces $f\oo{0}=f\oo{P}$, so
\begin{equation}
A=A\cos\oo{\frac{P}{\sqrt{\norm{\lambda}}}}+B\sin\oo{\frac{P}{\sqrt{\norm{\lambda}}}}.\label{AppendixLemma1,3}
\end{equation}
Also, the requirement that $\int_{0}^{P}{f\,dx}=0$ forces
\begin{equation}
A\sin\oo{\frac{P}{\sqrt{\norm{\lambda}}}}-B\cos\oo{\frac{P}{\sqrt{\norm{\lambda}}}}=-B.\label{AppendixLemma1,4}
\end{equation}
Combining $\oo{\ref{AppendixLemma1,3}}$ and $\oo{\ref{AppendixLemma1,4}}$,
\[
A^{2}=A^{2}\cos\oo{\frac{P}{\sqrt{\norm{\lambda}}}}+AB\sin\oo{\frac{P}{\sqrt{\norm{\lambda}}}}\text{  and  }B^{2}=B^{2}\cos\oo{\frac{P}{\sqrt{\norm{\lambda}}}}-AB\sin\oo{\frac{P}{\sqrt{\norm{\lambda}}}},
\]
meaning that
\[
A^{2}+B^{2}=\oo{A^{2}+B^{2}}\cos\oo{\frac{P}{\sqrt{\norm{\lambda}}}}.
\]
We conclude
\[
\frac{P}{\sqrt{\norm{\lambda}}}=2n\pi
\]
for some $n\in\mathbb{Z}\backslash\left\{0\right\}$ to be determined. 
Hence
\begin{equation}
f\oo{x}=A\cos\oo{\frac{2n\pi x}{P}}+B\sin\oo{\frac{2n\pi x}{P}}.\label{AppendixLemma1,5}
\end{equation}
A quick calculation yields
\[
\int_{0}^{P}{f^{2}\,dx}=\oo{\frac{A^{2}+B^{2}}{2}}P,\,\,\text{and}\,\,\int_{0}^{P}{f_{x}^{2}\,dx}=\oo{\frac{2n\pi}{P}}^{2}\oo{\frac{A^{2}+B^{2}}{2}}P.
\]

Hence for any of our extremal functions $f$,
\[
\frac{\int_{0}^{P}{f^{2}\,dx}}{\int_{0}^{P}{f_{x}^{2}\,dx}}=\oo{\frac{P}{2n\pi}}^{2}\leq\frac{P^{2}}{4\pi^{2}},
\]
with equality if and only if $n=1$. Thus our constrained function $f$ that maximises the ratio $\frac{\int_{0}^{P}{f^{2}\,dx}}{\int_{0}^{P}{f_{x}^{2}\,dx}}$ is given by
\[
f\oo{x}=A\cos\oo{\frac{2\pi}{P}x}+B\sin\oo{\frac{2\pi}{P}x},
\]
with 
\[
\int_{0}^{P}{f^{2}\,dx}\leq\frac{P^{2}}{4\pi^{2}}\int_{0}^{P}{f_{x}^{2}\,dx}
\]
amongst all continuous and $P-$periodic functions with $\int_{0}^{P}{f\,dx}=0$.
\end{proof}
\begin{lemma}\label{AppendixLemma2}
Let $f:\mathbb{R}\rightarrow\mathbb{R}$ be an absolutely continuous and periodic function of period $P$. Then, if $\int_{0}^{P}{f\,dx}=0$ we have
\[
\llll{f}_{\infty}^{2}\leq\frac{P}{2\pi}\int_{0}^{P}{f_{x}^{2}\,dx}.
\]
\end{lemma}
\begin{proof}
Since $\int_{0}^{P}{f\,dx}=0$ and $f$ is $P-$periodic we conclude that there exists distinct $0\leq p<q<P$ such that 
\[
f\oo{p}=f\oo{q}=0.
\]
Next, the fundamental theorem of calculus tells us that for any $x\in\oo{0,P}$,
\[
\frac{1}{2}\cc{f\oo{x}}^{2}=\int_{p}^{x}{ff_{x}\,dx}=\int_{q}^{x}{ff_{x}\,dx}.
\]
Hence
\begin{align*}
&\cc{f\oo{x}}^{2}=\int_{p}^{x}{ff_{x}\,dx}-\int_{qx}^{q}{ff_{x}\,dx}\leq\int_{p}^{q}{\norm{ff_{x}}\,dx}\leq\int_{0}^{P}{\norm{ff_{x}}\,dx}\\
&\leq\oo{\int_{0}^{P}{f^{2}\,dx}\cdot\int_{0}^{P}{f_{x}^{2}\,dx}}^{\frac{1}{2}}\leq\frac{P}{2\pi}\int_{0}^{P}{f_{x}^{2}\,dx},
\end{align*}
where the last step follows from Lemma $\ref{AppendixLemma1}$. We have also utilised H\"{o}lder's inequality with $p=q=2$.
\end{proof}

\begin{lemma}[\cite{DKS}, Lemma $2.4$]\label{AppendixLemma3}
Let $\gamma:\mathbb{S}^{1}\rightarrow\mathbb{R}^{2}$ be a smooth closed curve. Let $\phi:\mathbb{S}^{1}\rightarrow\mathbb{R}$ be a sufficiently smooth function. Then for any $l\geq 2,K\in\mathbb{N}$ and $0\leq i<K$ we have
\begin{equation}
L^{i+1-\frac{1}{l}}\oo{\intcurve{\oo{\phi}_{s^{i}}^{2}}}^{\frac{1}{l}}\leq c\oo{K}L^{\frac{1-\alpha}{2}}\oo{\intcurve{\phi^{2}}}^{\frac{1-\alpha}{2}}\llll{\phi}_{K,2}^{\alpha}.\label{AppendixLemma3,1}
\end{equation}
Here $\alpha=\frac{i+\frac{1}{2}-\frac{1}{l}}{K}$, and
\[
\llll{\phi}_{K,2}:=\sum_{j=0}^{K}L^{j+\frac{1}{2}}\oo{\intcurve{\oo{\phi}_{s^{j}}^{2}}}^{\frac{1}{2}}.
\]
In particular, if $\phi=\kappa-\bar{\phi}$, then
\begin{equation}
L^{i+1-\frac{1}{l}}\oo{\intcurve{\oo{k-\bar{k}}_{s^{i}}^{2}}}^{\frac{1}{l}}\leq c\oo{K}\oo{K_{osc}}^{\frac{1-\alpha}{2}}\llll{k-\bar{k}}_{K,2}^{\alpha}.\label{AppendixLemma3,2}
\end{equation}
\end{lemma}
\begin{proof}
The proof is identical to that of Lemma $2.4$ from \cite{DKS} and is of a
standard interpolative nature. Note that although we use $k-\bar{k}$ in the
identity (as opposed to \cite{DKS} where $k\nu$ is used).
\end{proof}
\begin{lemma}[Proposition $2.5$, \cite{DKS}]\label{AppendixLemma4}
Let $\gamma:\mathbb{S}^{1}\rightarrow\mathbb{R}^{2}$ be a smooth closed curve. Let $\phi:\mathbb{S}^{1}\rightarrow\mathbb{R}$ be a sufficiently smooth function. Then for any term $P_{\nu}^{\mu}\oo{\phi}$ (where $P_{\nu}^{\mu}\oo{\cdot}$ denotes the same $P$-style notation used in for example \cite{DKS}) with $\nu\geq2$ which contains only derivatives of $\kappa$ of order at most $K-1$, we have
\begin{equation}
\intcurve{\norm{P_{\nu}^{\mu}\oo{\phi}}}\leq c\oo{K,\mu,\nu}L^{1-\mu-\nu}\oo{L\intcurve{\phi^{2}}}^{\frac{\nu-\eta}{2}}\llll{\phi}_{K,2}^{\eta}.\label{AppendixLemma4,1}
\end{equation}
In particular, for $\phi=\kappa-\bar{\kappa}$ we have the estimate
\begin{equation}
\intcurve{\norm{P_{\nu}^{\mu}\oo{\kappa-\bar{\kappa}}}}\leq c\oo{K,\mu,\nu}L^{1-\mu-\nu}\oo{K_{osc}}^{\frac{\nu-\eta}{2}}\llll{\kappa-\bar{\kappa}}_{K,2}^{\eta}\label{AppendixLemma4,2}
\end{equation}
where $\eta=\frac{\mu+\frac{\nu}{2}-1}{K}$.
\end{lemma}
\begin{proof}
Using H\"{o}lder's inequality and Lemma $\ref{AppendixLemma3}$ with $K=\nu$, if $\sum_{j=1}^{\nu}i_{j}=\mu$ we have
\begin{align}
&\intcurve{\norm{\phi_{s^{i_{1}}}\star\cdots\star\phi_{s^{i_{\nu}}}}}\nonumber\\
&\leq\prod_{j=1}^{\nu}\oo{\intcurve{\phi_{s^{i_{j}}}^{\nu}}}^{\frac{1}{\nu}}=L^{1-\mu-\nu}\prod_{j=1}^{\nu}L^{i_{j}+1-\frac{1}{\nu}}\oo{\intcurve{\phi_{s^{i_{j}}}^{\nu}}}^{\frac{1}{\nu}}\nonumber\\
&\leq c\oo{K,\mu,\nu}L^{1-\mu-\nu}\prod_{j=1}^{\nu}\oo{L\intcurve{\phi^{2}}}^{\frac{1-\alpha_{j}}{2}}\llll{\phi}_{K,2}^{\alpha_{j}}\label{AppendixLemma4,3}
\end{align}
where $\alpha_{j}=\frac{i_{j}+\frac{1}{2}-\frac{1}{\nu}}{K}$. Now
\[
\sum_{j=1}^{\nu}\alpha_{j}=\frac{1}{K}\sum_{j=1}^{\nu}\oo{i_{j}+\frac{1}{2}-\frac{1}{\nu}}=\frac{\mu+\frac{\nu}{2}-1}{K}=\eta,
\]
ans so substituting this into $\oo{\ref{AppendixLemma4,3}}$ gives the first inequality of the lemma. It is then a simple matter of substituting $\phi=\kappa-\bar{\kappa}$ into this result to prove statement $\oo{\ref{AppendixLemma4,2}}$.
\end{proof}
\begin{lemma}[\cite{DKS}]\label{AppendixLemma5}
Let $\gamma:\mathbb{S}^{1}\rightarrow\mathbb{R}^{2}$ be a smooth closed curve and $\phi:\mathbb{S}^{1}\rightarrow\mathbb{R}$ a sufficiently smooth function.
Then for any term $P_{\nu}^{\mu}\oo{\phi}$ with $\nu\geq2$ which contains only derivatives of $\kappa$ of order at most $K-1$, we have for any $\varepsilon>0$
\begin{equation}
\intcurve{\norm{P_{\nu}^{\mu,K-1}\oo{\phi}}}\leq c\oo{K,\mu,\nu}L^{1-\mu-\nu}\oo{L\intcurve{\phi^{2}}}^{\frac{\nu-\eta}{2}}\oo{L^{2K+1}\intcurve{\phi_{s^{K}}^{2}}+L\intcurve{\phi^{2}}}^{\frac{\eta}{2}}.\label{AppendixLemma5,1}
\end{equation}
Moreover if $\mu+\frac{1}{2}\nu<2K+1$ then $\eta<2$ and we have for any $\varepsilon>0$
\begin{equation}
\intcurve{\norm{P_{\nu}^{\mu,K-1}\oo{\phi}}}\leq\varepsilon\intcurve{\phi_{s^{K}}^{2}}+c\cdot\varepsilon^{-\frac{\eta}{2-\eta}}\oo{\intcurve{\phi^{2}}}^{\frac{\nu-\eta}{2-\eta}}+c\oo{\intcurve{\phi^{2}}}^{\mu+\nu-1}.\label{AppendixLemma5,2}
\end{equation}
In particular, for $\phi=\kappa-\bar{\kappa}$, we have the estimate
\[
\intcurve{\norm{P_{\nu}^{\mu}\oo{k-\bar{k}}}}\leq c\oo{K,\mu,\nu}L^{1-\mu-\nu}\oo{K_{osc}}^{\frac{\nu-\eta}{2}}\oo{L^{2K+1}\intcurve{\oo{k-\bar{k}}_{s^{K}}^{2}}}^{\frac{\eta}{2}}.
\]
Here, as before, $\eta=\frac{\mu+\frac{\nu}{2}-1}{K}$. 
\end{lemma}
\begin{proof}
Combining the previous lemma with the following standard interpolation inequality from that follows from repeated applications of Lemma $\ref{AppendixLemma0}$ (and is also found in \cite{Aubin1})
\[
\llll{\phi}_{K,2}^{2}\leq c\oo{K}\oo{L^{2K+1}\intcurve{\phi_{s^{K}}^{2}}+L\intcurve{\phi^{2}}}
\]
yields the identity $\oo{\ref{AppendixLemma5,1}}$ immediately. To prove $\oo{\ref{AppendixLemma5,2}}$ we simply combine $\oo{\ref{AppendixLemma5,1}}$ with the Cauchy-Schwarz identity. The final identity of the Lemma follow by letting $\phi=\kappa-\bar{\kappa}$ in $\oo{\ref{AppendixLemma5,1}}$ and combining this with the identity
\begin{equation}
K_{osc}\leq L\oo{\frac{L^{2}}{4\pi^{2}}}^{K}\intcurve{\oo{\kappa-\bar{\kappa}}_{s^{K}}^{2}}= c\oo{K}L^{2K+1}\intcurve{\oo{\kappa-\bar{\kappa}}_{s^{K}}^{2}},\label{AppendixLemma5,3}
\end{equation}
which is a direct consequence of applying Lemma $\ref{AppendixLemma1}$ $\oo{p+1}$ times repeatedly.
\end{proof}

\begin{theorem}[\cite{Breuning1}, Theorem $1.1$]\label{AppendixTheorem1}
Let $q\in\mathbb{R}^{n}$, $m,p\in\mathbb{N}$ with $p>m$. Additionally, let $\mathcal{A},\mathcal{V}>0$ be some fixed constants. Let $\mathfrak{T}$ be the set of all mappings $f:\Sigma:\rightarrow\mathbb{R}^{n}$ with the following properties:
\begin{itemize}
\item $\Sigma$ is an $m$-dimensional, compact manifold (without boundary)

\item $f$ is an immersion in $W^{2,p}\oo{\Sigma,\mathbb{R}^{n}}$ satisfying
\begin{align*}
\llll{A\oo{f}}_{p}&\leq\mathcal{A},\\
\text{vol}\oo{\Sigma}&\leq\mathcal{V},\text{  and  }\\
q&\in f\oo{\Sigma}.
\end{align*}
\end{itemize}
Then for every sequence $f^{i}:\Sigma^{i}\rightarrow\mathbb{R}^{n}$ in $\mathfrak{T}$ there is a subsequence $f^{j}$, a mapping $f:\Sigma\rightarrow\mathbb{R}^{n}$ in $\mathfrak{T}$ and a sequence of diffeomorphisms $\phi^{j}:\Sigma\rightarrow\Sigma^{j}$ such that $f^{j}\circ\phi^{j}$ converges in the $C^{1}$-topology to $f$.
\end{theorem}
\end{section}

\bibliographystyle{plain}

\bibliography{PFBib}

\begin{thebibliography}{10}

\bibitem{Aubin1}
T.~Aubin.
\newblock {\em Nonlinear analysis on manifolds: Monge-Ampere equations}, volume
  252.
\newblock Springer, 1982.

\bibitem{Baker}
C.~Baker.
\newblock {\em The mean curvature flow of submanifolds of high codimension}.
\newblock PhD thesis, Australian National University, 2011.

\bibitem{Blatt2010}
S.~Blatt.
\newblock Loss of convexity and embeddedness for geometric evolution equations
  of higher order.
\newblock {\em J. Evol. Equ.}, 10(1):21--27, 2010.

\bibitem{Breuning1}
P.~Breuning.
\newblock Immersions with bounded second fundamental form.
\newblock {\em arXiv preprint arXiv:1201.4562}, 2012.

\bibitem{DGK15}
D.~Daners, J.~Gl\"uck, and J.~Kennedy.
\newblock Eventually positive semigroups of linear operators.
\newblock Submitted.

\bibitem{DKS}
G.~Dziuk, E.~Kuwert, and R.~Sch{\"a}tzle.
\newblock {Evolution of elastic curves in $\R^n$: existence and computation}.
\newblock {\em SIAM J. Math. Anal.}, 33(5):1228--1245, 2002.

\bibitem{EG97}
C.~Elliott and H.~Garcke.
\newblock {Existence results for diffusive surface motion laws}.
\newblock {\em Adv. Math. Sci. Appl.}, 7(1):467--490, 1997.

\bibitem{EllPaa2001}
C.~Elliott and S.~Maier-Paape.
\newblock Losing a graph with surface diffusion.
\newblock {\em Hokkaido Math. J.}, 30:297--305, 2001.

\bibitem{EscherIto2005}
J.~Escher and K.~Ito.
\newblock {Some dynamic properties of volume preserving curvature driven
  flows}.
\newblock {\em Math. Ann.}, 333(1):213--230, 2005.

\bibitem{FGG08}
A.~Ferrero, F.~Gazzola, and H.-Ch. Grunau.
\newblock {Decay and eventual local positivity for biharmonic parabolic
  equations}.
\newblock {\em Dyn. Syst.}, 21(4):1129--1157, 2008.

\bibitem{Fife}
P.~Fife.
\newblock Models for phase separation and their mathematics.
\newblock {\em Electron. J. Differential Equations}, 2000(48):1--26, 2000.

\bibitem{Hamilton2}
M.~Gage and R.~Hamilton.
\newblock The heat equation shrinking convex plane curves.
\newblock {\em J. Differential Geom.}, 23(1):69--96, 1986.

\bibitem{GG08}
F.~Gazzola and H.-Ch. Grunau.
\newblock {Eventual local positivity for a biharmonic heat equation in $\R^n$}.
\newblock {\em Discrete Contin. Dyn. Syst., Ser. S}, 1:83--87, 2008.

\bibitem{GG09}
F.~Gazzola and H.-Ch. Grunau.
\newblock {Some new properties of biharmonic heat kernels}.
\newblock {\em Nonlinear Anal.}, 70(8):2965--2973, 2009.

\bibitem{privatecommsgiga}
Y.~Giga.
\newblock personal communication.

\bibitem{GI1998}
Y.~Giga and K.~Ito.
\newblock On pinching of curves moved by surface diffusion.
\newblock {\em Comm. Appl. Anal.}, 2(3):393--406, 1998.

\bibitem{GI1999}
Y.~Giga and K.~Ito.
\newblock Loss of convexity of simple closed curves moved by surface diffusion.
\newblock In {\em {Topics in Nonlinear Analysis, The Herbert Amann anniversary
  volume (eds. J. Escher and G. Simonett)}}, volume~35 of {\em {Progress in
  Nonlinear Differential Equations and Their Applications}}, pages 305--320.
  Birkh{\"a}user, 1999.

\bibitem{grayson1989}
M.~Grayson.
\newblock Shortening embedded curves.
\newblock {\em Ann. of Math. (2)}, pages 71--111, 1989.

\bibitem{Mantegazza2}
C.~Mantegazza and L.~Martinazzi.
\newblock A note on quasilinear parabolic equations on manifolds.
\newblock {\em Ann. Scuola Norm. Sup. Pisa Cl. Sci}, 2011.

\bibitem{privatecommsmayer}
U.~Mayer.
\newblock personal communication.

\bibitem{Mullins1}
W.~Mullins.
\newblock Two-dimensional motion of idealized grain boundaries.
\newblock {\em J. Appl. Phys.}, 27(8):900--904, 2004.

\bibitem{Wkosc}
G.~Wheeler.
\newblock On the curve diffusion flow of closed plane curves.
\newblock {\em Ann. Mat. Pura Appl. (4)}, 192(5):931--950, 2013.

\end{thebibliography}

\end{document}